\documentclass[11pt]{article}
\usepackage{amsfonts}
\usepackage{amsfonts,latexsym}
\usepackage{amsmath}
\usepackage{amsthm}
\usepackage{amssymb}
\usepackage{mathrsfs}
\usepackage{hyperref}
\usepackage{xcolor}
\usepackage[english]{babel}        % 基础语言支持
\usepackage[]{cite}
\usepackage{amsmath}
\allowdisplaybreaks[4]
\usepackage{indentfirst}
\usepackage{geometry}
\newtheorem{theorem}{\indent Theorem}[section]
\newtheorem{proposition}[theorem]{\indent Proposition}
\newtheorem{definition}[theorem]{\indent Definition}
\newtheorem{lemma}[theorem]{\indent Lemma}

\newtheorem{example}{Example}[section]

\begin{document}

\title{\bf Uniform measure attractors of the distribution-dependent 2D stochastic Navier-Stokes equations driven by nonlinear noise}

\author{{Jiangwei Zhang$^\text{a}$, \quad  Juntao Wu$^\text{b}\footnote{Corresponding author.}$
	}\\
	{ \small\textsl{$^\text{a}$ Institute of Applied Physics and Computational Mathematics, }}\\
	{ \small \textsl{Beijing 100088,  P.R. China}}\\
	{ \small\textsl{$^\text{b}$ School of Mathematics and Statistics, Wuhan University, }}\\ 
	{ \small \textsl{Wuhan, Hubei 430072, P.R. China}}
}
\footnotetext{
	\emph{E-mail addresses}: jwzhang0202@yeah.net (J. Zhang), wu\_jun\_tao@163.com (J. Wu).
}
\date{}

%%%%%%%%%%%%%%%%%%%%%%%%

\renewcommand{\theequation}{\arabic{section}.\arabic{equation}}
\numberwithin{equation}{section}

\maketitle
 
%%%%%%%%%%%%%%%%%%%%%
\begin{abstract} In this paper, we investigate the uniform measure attractors of the distribution-dependent nonautonomous 2D stochastic Navier-Stokes equations driven by nonlinear noise and subject to almost periodic external forcing. Owing to the distribution-dependent structure and the almost periodicity of the external forcing, the resulting solution process becomes an inhomogeneous Markov process, presenting significant analytical challenges.  To overcome these difficulties, we propose sufficient conditions on the time-dependent external forcing and distribution-dependent nonlinear terms, and develop novel analytical estimates. As a result, we establish the existence and uniqueness of uniform measure attractors for the system. Notably, the joint continuity of the family of processes is achieved without relying on the Feller property of the distribution law operators.

\medskip

\end{abstract}

{\hspace*{2mm}  AMS Subject Classification: 35B40, 35B41, 37B55, 37L55, 37L30.}
%%%%%%%%%%%%%%%%%%%%%%%%%%%%%%%%%%%%%%%%%%%%%%%%%%%%%%%%%%%%%%%%%%%%%%%%%%%%%%%

\section{Introduction}
In this paper, we are concerned with the asymptotic behavior of solutions for the following distribution-dependent nonautonomous 2D stochastic Navier-Stokes equations with almost periodic external forcing and nonlinear noise:
\begin{align}\label{NSE-dis1.1}
	\left\{
	\begin{aligned}
		&du(t)-\nu\Delta u(t)dt+(u(t)\cdot \nabla )u(t)dt+\nabla pdt
		=g(t,x)dt\\
		&\qquad +f\left(x, u (t), \mathscr{L}_{u(t)} \right) dt+\varepsilon\sum_{k=1}^{\infty}
		\left(h(t,x)+\kappa(x) \sigma_{k}\left(t,u (t), \mathscr{L}_{u (t)}\right)\right) dW_k(t),\quad t>\tau,\\
		&{\rm div} \, u=0,\quad
		\text{ on }\,\, \mathcal{O}\times(\tau,\infty),
	\end{aligned}
	\right.
\end{align}
with the initial-boundary conditions
\begin{align}\label{NSE-dis1.2}
	 u(t,x)|_{(\tau,\infty)\times\partial \mathcal{O}}=0, \quad u(\tau,x)=u_\tau(x),\quad  x\in\mathcal{O},
\end{align}
where $\mathcal{O}\subset\mathbb{R}^2$ is an open bounded domain with smooth boundary $\partial \mathcal{O}$,  $u$ and $p$ denote the velocity field and pressure of fluid, $\mathscr{L}_{u(t)}$ represents the probability distribution of ${u(t)}$;  $\nu>0$ is  the
viscosity constant, $\varepsilon\in (0,1]$ denotes the intensity of noise, $\kappa:=\kappa(x)\in W^{1,\infty}(\mathcal{O})$, the time-dependent external forcing terms $g(t) = g(t,x)$ and $h(t) = h(t,x)$ are almost periodic in time $t$. Moreover,
$f$ and $\sigma_{k}$ are nonlinear functions which will be given later. $\{W_k\}_{k \in \mathbb{N}}$ is a sequence of independent two-sided real-valued Wiener processes defined on a complete filtered probability space $(\Omega, \mathscr{F}, \{\mathscr{F}_t\}_{t \in \mathbb{R}}, \mathbb{P})$ satisfying the usual condition.
%(SNSEs)

The McKean-Vlasov stochastic differential equations (MVSDEs) constitutes a mean-field model characterizing the weak convergence limits of large-scale interacting particle systems. It has been extensively applied in numerous fields such as aerospace engineering, plasma physics and statistical mechanics, providing a rigorous mathematical framework that bridges microscopic stochastic dynamics and emergent macroscopic behavior, see \cite{Leray1934,Bensoussan-1973,Chow-1978} and the references therein. 
In recent years, numerous scholars investigated the well-posedness and dynamical behavior of the solutions of MVSDEs, see e.g., \cite{	McKean-1966,Vlasov-1968,Dawson-1995,Sznitman-1991,Prato2009,Carmona-2018,Rockner-2021,WangFY-2023,Fan-2022,WangFy-2018, H1}.
We notice that a fundamental property of MVSDEs lies in that both the drift and diffusion coefficients depend not only on the system state $u(t)$ but also on its probability distribution $\mathscr{L}_{u(t)}$. This structure has motivated the development of distribution-dependent stochastic partial differential equations, including the stochastic abstract fluid equations with mean-field interactions \cite{CQT-2025}.
In this work, we devote to studying the uniform measure attractors for the distribution-dependent 2D stochastic Navier-Stokes equations \eqref{NSE-dis1.1}, which emerges as the mean-field limit of M-interacting hydrodynamic flows. The model incorporates dependence on the law of the solution, extending classical fluid models to settings with non-local interactions and measure-valued nonlinearities.

When the nonlinear terms $f$ and $\sigma_k$ in \eqref{NSE-dis1.1} are independent of the law $\mathscr{L}_{u(t)}$, the equations \eqref{NSE-dis1.1} reduce to the classical stochastic Navier-Stokes equations with nonlinear noise. In general, for the systems with linear or additive noise, the asymptotic behavior of solutions can be characterized by pathwise pullback random attractors \cite{ra,glatt,Prato-2008,Crauel-1994,Prato-2002,Temam1997}. However, for the case of nonlinear noise, it becomes necessary to employ the framework of measure attractors. For detailed theoretical foundations and applications of measure attractors, we refer to \cite{Schmalfuss-1991,cap,morimoto1992,Cutland-1998,schmalfuss1999,crauel2008,schmalfuss1997}.  

More recently, the authors of \cite{LDS-JDE-2024} introduced the concept of pullback measure attractors for nonautonomous dynamical systems and applied this abstract framework to reaction-diffusion equations subject to deterministic nonautonomous forcing on thin domains. Following this development, numerous studies have extended the analysis of pullback measure attractors to various types of stochastic differential equations, as evidenced in \cite{LZWH-BMMSS,LML-QTDS-2024,MLZ-AMO-2024,MLZ-Procd-2024}. It should be noted, however, that the solution processes in these studies are homogeneous Markov processes. If the deterministic external forcing is almost periodic, then the corresponding solution process becomes inhomogeneous Markov process. To address this scenario, the authors of \cite{li2024b, LZH-AML} proposed the concept of uniform measure attractors and applied it to stochastic (tamed) Navier-Stokes equations.
Nevertheless, all the aforementioned works are restricted to stochastic partial differential equations that do not depend on the distribution of the solutions.

%In particular, the well-posedness and asymptotic behavior of solutions to this class of equations have been extensively studied, see  and the references therein. external force

For the distribution-dependent stochastic systems, the authors of  \cite{SSLW-2024,SSLd-2024} established the existence and uniqueness of pullback measure attractors for the McKean-Vlasov stochastic reaction diffusion equations and the McKean-Vlasov  stochastic delay lattice systems. Moreover, the first author and collaborators investigated the well-posedness and pullback measure attractors for the distribution-dependent nonautonomous 2D stochastic Navier-Stokes equations \eqref{NSE-dis1.1} with deterministic external forcing.
However, to the best of our knowledge, no results are currently available for the distribution-dependent 2D stochastic Navier-Stokes equations with  time-dependent almost periodic external forcing, such as those described by systems \eqref{NSE-dis1.1}. 
More precisely, 
in this work, we aim to study the existence and uniqueness of uniform measure attractors for such systems.

To outline the main challenges addressed in this study, we denote by $P^{(g,h)}(\tau, t)$ the transition operator associated with the solution of \eqref{NSE-dis1.1}, and let $P_*^{(g,h)}(\tau, t)\mu$ represent the law of the solution
of \eqref{NSE-dis1.1} with initial law $\mu\in \mathcal{P}_4(H)$ at initial time $\tau$, where the definitions of $H$ and $\mathcal{P}_4(H)$ can be found in Sections \ref{NS-Sec2} and \ref{NS-Sec3}, respectively. For the distribution-dependent nonautonomous 2D stochastic Navier-Stokes equations with almost periodic external forcing \eqref{NSE-dis1.1}, we know that $P_*^{(g,h)}(\tau, t) $ is not the dual operator of $P^{(g,h)}(\tau, t) $. Specifically, the duality relation
\begin{align*}
	\int_{H} P_{\tau, t}\phi (x) d\mu (x)
	\neq   \int_{H}\phi (x)  d (P^*_{\tau, t}\mu) (x)
\end{align*} 
fails to hold for any $\mu\in {\mathcal{P}}_4 (H)$ and bounded Borel functions $\phi:  H \to \mathbb{R}$, because the dual operator does not form a flow (cf. \cite{WangFy-2018}). This is the main reason why we employ law operators instead of dual operators for distribution-dependent equations. The failure of this duality introduces significant technical obstacles. To establish the existence of uniform measure attractors, it is necessary to demonstrate the weak continuity of the family of processes $\{P_*^{(g,h)}(\tau, t)\}_{\tau\leq t}$ on $(\mathcal {P}_4 (X), d_{\mathcal{P}(X)})$, which in turn ensures the joint continuity of the family of processes. In the distribution-independent setting, the continuity of $\{P_*^{(g,h)}(\tau, t)\}_{\tau\leq t}$ follows from the Feller property of $\{P^{(g,h)}(\tau, t)\}_{\tau\leq t}$, due to the duality between these operators. However, this argument no longer holds in the distribution-independent case described by \eqref{NSE-dis1.1}, necessitating alternative approaches. In particular, the long-time uniform estimates in $L^2(\Omega,V)$ for the solution of \eqref{NSE-dis1.1} cannot be directly obtained.

To establish the existence and uniqueness of uniform measure attractors for the distribution-dependent nonautonomous 2D stochastic Navier-Stokes equations \eqref{NSE-dis1.1} with almost periodic external forcing and nonlinear noise, we propose the following strategy:

$(i)$ The joint continuity of the family of processes $\{P_*^{(g,h)}(\tau, t)\}_{\tau\leq t}$ will be established on the subspace $(\mathbb{B}_{\mathcal{P}_{4}(X)}(r),d_{\mathcal{P}(H)})$, rather than on the entire space $(\mathcal{P}_4(H),d_{\mathcal{P}(H)})$. 

$(ii)$ The asymptotic compactness of the family $\{P_*^{(g,h)}{(\tau, t)}\}_{\tau\leq t}$ shall be verified by deriving the long-time uniform estimates of solutions on $L^2(\Omega,V)$ with the appropriate weight.

The structure of this paper is as follows. Section \ref{NS-Sec2} introduces the necessary notations and abstract concepts pertaining to spaces of probability measures and uniform measure attractors. In Section \ref{NS-Sec3}, we establish sufficient conditions for the existence of uniform measure attractors for \eqref{NSE-dis1.1} and present corresponding well-posedness results under these conditions. Section \ref{NS-Sec4} is devoted to proving the long-time uniform estimates of the solution, which are essential for establishing the existence and uniqueness of uniform measure attractors for \eqref{NSE-dis1.1}. Finally, in Section \ref{ex-pull-NS}, we demonstrate the existence and uniqueness of uniform measure attractors.

\section{Preliminaries}\label{NS-Sec2}
In this section, we introduce some basic theories of uniform measure attractors for a family of processes acting on the space of probability measures over a Banach space. Additionally, the structure of uniform measure attractors is given.

\subsection{Basic probability measure spaces}
Let $X$ be a separable Banach space with norm $\|\cdot\|_X$. Denote by $C_b(X)$ the space of bounded continuous functions $\phi:X\rightarrow \mathbb{R}$ equipped with the supremum norm 
$
\|\phi\|_{C_b}=\sup_{x\in X}|\phi(x)|.
$
 Let $L_b(X)$ denote the space of bounded Lipschitz functions on $X$, consisting of all functions $\phi\in C_b(X)$ such that
\begin{align*}
	\|\phi\|_{\text{Lip}} = \sup\limits_{x_{1},x_{2} \in X,x_{1} \neq x_{2}}\ \frac{\left| \phi\left( x_{1} \right) - \phi\left( x_{2} \right) \right|}{{\| x_{1} - x_{2}\|}_{X}} < \infty.
\end{align*}
The space $L_b(X)$ is endowed with the norm
$$
\|\phi\|_{L_b}=\|\phi\|_{C_b}+\|\phi\|_{\text{Lip}} .
$$
%\text{Lip}(\phi):

Let $\mathcal{P}(X)$ be the space of probability measures on $\left(X,\mathcal{B}(X)\right)$, here $\mathcal{B}(X)$ represents the Borel $\sigma$-algebra of $X$. For given $\phi\in C_b(X)$ and $\mu\in \mathcal{P}(X)$, we write
$$
(\phi,\mu)=\int_{X} \phi(x)\mu(dx).
$$
Define a metric of $\mathcal{P}(X)$ by
$$
d_{\mathcal{P}(X)}\left( \mu_{1},\mu_{2} \right) = \sup\limits_{{\phi \in L_{b}{(X)}},\,{\|\phi\|_{L_b} \leq 1}}\ \left| \left( \phi,\mu_{1} \right) - \left( \phi,\mu_{2} \right) \right|,\quad\forall \mu_{1},\mu_{2} \in \mathcal{P}(X).
$$
It follows that $\left(\mathcal{P}(X),d_{\mathcal{P}(X)}\right)$ is a Polish space. Furthermore, a sequence  $\{\mu_n\}_{n=1}^\infty\subset \mathcal{P}(X)$ is weakly convergent to $\mu \in \mathcal{P}(X)$, if for every $\phi\in C_b(X)$, it holds
$
\lim_{n\rightarrow \infty} \left(\phi,\mu_n\right)=\left(\phi,\mu\right).
$

For every $p\geq 1$, the $p$-Wasserstein space $\mathcal P_p \left( X \right)$ on $X$ is defined as
$$
\mathcal P_p \left( X \right) = \left\{ {\mu  \in \mathcal
	P\left( X \right):
	\int_X {\|x\|_X^p \mu \left( {dx} \right) <    \infty } } \right\},
$$
and $p$-Wasserstein distance $\mathbb{W}_p$ is given by
$$
\mathbb{W}_p ( \mu , \nu ) =
\inf\limits_{ \pi \in \mathscr{C} ( \mu, \nu ) }
\Big (
\int_{ X \times X}
\| x-y  \|_X^p \pi (dx, dy)
\Big )^{ \frac{1}{p} },\quad
\forall \mu, \nu \in \mathcal{P}_p ( X ),
$$
where $ \mathscr{C} ( \mu, \nu ) $ is the  set of all couplings of $\mu$ and $\nu$. Then, $(\mathcal P_p(X), \mathbb{W}_p )$ is a
Polish space.

Given $r>0$, we define the ball $\mathbb{B}_{\mathcal{P}_{p}(X)}(r)$ by
\begin{align*}
	\mathbb{B}_{\mathcal{P}_{p}(X)}(r) = \left\{ \mu \in \mathcal{P}_{p}(X):\left( \int_X  \|x\|_{X}^{p}\mu(dx) \right)^{\frac{1}{p}} \leq r \right\}.
\end{align*}
A  subset $\mathfrak{B}\subseteq {\mathcal P}_p \left( X \right)$ is
bounded  if there exists $r>0$ such that
$\mathfrak{B}\subseteq \mathbb{B}_{\mathcal{P}_{p}(X)}(r)$. If   $\mathfrak{B}$ is bounded in   $ {\mathcal P}_p ( X )$,
then we set
$$
\|\mathfrak{B}\|_{{\mathcal P}_p ( X )}
=
\sup_{
	\mu \in \mathfrak{B}
} \left(\int_X \|x\|_X^p \mu
( {dx}  )\right)^{\frac{1}{p}}.
$$

Note that 
$(\mathcal {P}_p (X), d_{\mathcal{P}(X)})$ is not complete. For every $r>0$, since $\mathbb{B}_ {\mathcal P_p(X)} (r)$ is a closed subset of  $\mathcal{P}(X)$ with respect to the metric  $d_{\mathcal{P}(X)}$, we know that the space
$(\mathbb{B}_ {\mathcal P_p(X)} (r), \ d_{\mathcal{P}(X)})$
is complete. In addition, the Hausdorff semi-distance between two nonempty subsets $Y$ and $Z$  of $\mathcal{P}_p(X)$ is defined by
\[
d_{\mathcal{P}_p(X)}(Y, Z) = \sup_{y \in Y} \inf_{z \in Z} d_{\mathcal{P}(X)}(y, z), \quad  \forall\, Y, Z \subseteq \mathcal{P}_p(X).
\]

%For any given $\varepsilon>0$, we define the $\varepsilon$-neighborhood of set  $\widetilde{B}\subset\mathcal{P}_{p}(X)$ as follows
%\begin{align*}
%	\mathcal{N}_{\varepsilon}(\widetilde{B}) = \left\{ \mu \in \mathcal{P}_{p}(X):d_{\mathcal{P}_{p}(X)}(\mu,\widetilde{B}) < \varepsilon \right\}.
%\end{align*}
\subsection{Abstract theory of uniform measure attractors}
 In this subsection, we introduce the family of processes with skew product semi-flow. Assume that \( g_0(t) \) and \( h_0(t) \) are almost periodic functions in \( t \in \mathbb{R} \) with values in \( X \). We further define \( C_b(\mathbb{R}, X) \) as the space of bounded continuous functions on \( \mathbb{R} \) with norm
\[
\|\varphi\|_{C_b(\mathbb{R}, X)} = \sup_{t \in \mathbb{R}} \|\varphi(t)\|_X, \text{\, for \, } \varphi\in C_b(\mathbb{R}, X).
\]
Since the almost periodic function is
bounded and uniformly continuous on $\mathbb{R}$, it follows that $g_0,h_0 \in C_b(\mathbb{R}, X)$.
By Bochner's criterion (see \cite{Levitan-1983}), whenever 
$g_0,h_0:\mathbb{R}\rightarrow X$ are almost periodic, then the sets of all translations 
\begin{equation*}
\{g_0(\cdot + s) : s \in \mathbb{R}\} \text{~~and~~}	\{h_0(\cdot + s) : s \in \mathbb{R}\} 
\end{equation*}
are precompact in \( C_b(\mathbb{R}, X) \) . 
Let \( \mathcal{H}(g_0) \) and \( \mathcal{H}(h_0) \) denote the closures of these translation sets in \( C_b(\mathbb{R}, X) \), respectively. Then, for any $g\in \mathcal{H}(g_0)$ and $h\in \mathcal{H}(h_0)$, $g$ and $h$ are almost periodic. In particular, we have $\mathcal{H}(g)=\mathcal{H}(g_0)$ and $\mathcal{H}(h)=\mathcal{H}(h_0)$.

For simplicity, we assume throughout this paper that  $g$ and $h$ are the almost periodic  functions in \( C_b(\mathbb{R}, X)\). The results presented here continue to hold for any forcing functions $g$ and $h$ whose hull \( \mathcal{H}(g_0) \times \mathcal{H}(h_0) \) is compact in \( C_b(\mathbb{R}, X) \times C_b(\mathbb{R}, X) \). For notational convenience, let  $\Sigma=\mathcal{H}(g_0) \times \mathcal{H}(h_0)$.

\begin{definition}\label{Def1}
	A family $U^{(g,h)}=\{U^{(g,h)}(t,\tau):t\geq \tau,\tau\in\mathbb{R}\}$ of mappings from $\mathcal {P}_p (X)$ to $\mathcal {P}_p (X)$ is called a process on $\mathcal {P}_p (X)$ with time symbol $(g,h)\in\Sigma$, if for all $\tau\in\mathbb{R}$ and $t\geq s\geq\tau$, the following conditions are satisfied:
	\begin{itemize}
		\item \( U^{(g,h)}(\tau, \tau) = I_{\mathcal{P}_p(X)} \), for $\tau \in \mathbb{R}$, here $I_{\mathcal{P}_p(X)}$ denotes the identity operator on \(\mathcal {P}_p (X)\);
		\item \( U^{(g,h)}(t, \tau) = U^{(g,h)}(t, s) \circ U^{(g,h)}(s, \tau) \) for \( t \geq s \geq \tau \in \mathbb{R}\).

	\item The family \( \{U^{(g,h)}(t, \tau)\}_{(g,h) \in \Sigma} \) is said to be jointly continuous if it is continuous in both \( \mathcal {P}_p (X) \) and \( \Sigma \).	\end{itemize}
\end{definition}

For any $s\in\mathbb{R}$, we denote the translation operator (or group) $T(s)$ on $\Sigma$ by
$$
T(s)(g,h)=(g(\cdot+s),h(\cdot+s)), \,\,\, \forall  (g,h)\in\Sigma.
$$ 
Obviously, the translation group $\{T(s)\}_{s\in\mathbb{R}}$ forms a continuous translation group on $\Sigma$ that leaves $\Sigma$ invariant:
$$T(s)\Sigma=\Sigma,\,\,\, \forall s\in\mathbb{R}.$$

Furthermore, for the processes $\left\{U^{(g,h)}(t,\tau)\right\}_{(g,h)\in\Sigma}$
and the translation group $\{T(s)\}_{s\in\mathbb{R}}$, we assume that they satisfies the following translation identity:
\[
U^{(g, h)}(t + s, \tau + s) = U^{T(s)(g, h)}(t, \tau), \,\,\, \forall s \in \mathbb{R}, \, t \geq \tau \text{ and } \tau \in \mathbb{R}.
\]

\begin{definition}
	A closed set $B\subset\mathcal{P}_p\left(X\right)$ is called a uniform absorbing set of the family of processes $\{U^{(g,h)}(t,\tau)\}_{(g,h)\in\Sigma}$ with respect to $(g,h)\in\Sigma$ if for any bounded set $\mathfrak{D}$ of $\mathcal{P}_p(X)$, there exists $T=T(\mathfrak{D},g_{0},h_{0})>0$ such that
	$$U^{(g,h)}\left(t,0\right)\mathfrak{D}\subseteq B,\quad\mathrm{for~all~}(g,h)\in\Sigma\mathrm{~and~}t\geq T.$$
\end{definition}

\begin{definition}
	The family of processes $\{U^{(g,h)}(t,\tau)\}_{(g,h)\in\Sigma}$ is called uniformly asymptotically compact in $\mathcal {P}_p (X)$ with respect to $\left(g,h\right)\in\Sigma$ if $\left\{U^{(g_n,h_n)}\left(t_n,0\right)\mu_n\right\}_{n=1}^\infty$ has a convergent subsequence in $\mathcal {P}_p (X)$ whenever $t_n\to+\infty$ and $\left(\mu_n,(g_n,h_n)\right)$ is bounded in $\mathcal{P}_p\left(X\right)\times\Sigma.$
\end{definition}

\begin{definition}
	A set $\mathscr{A}$ of $\mathcal {P}_p (X)$ is said to be a uniform measure attractor of the family of processes $\{U^{(g,h)}(t,\tau)\}_{(g,h)\in\Sigma}$ with respect to $(g,h)\in\Sigma$ if
	\begin{itemize}
		\item[$(i)$] \( \mathscr{A} \) is compact in \( \mathcal {P}_p (X) \);
		\item[$(ii)$] $\mathscr{A}$ is uniformly quasi-invariant, that is, for each $t\geq \tau\in \mathbb{R}$, one has 
		$$
		\mathscr{A}\subseteq \bigcup_{(g,h)\in\Sigma}{U^{(g,h)}(t,\tau)}\mathscr{A};
		$$
		\item[$(iii)$] $\mathscr{A}$ attracts every bounded set in $\mathcal {P}_p (X)$ uniformly with respect to $(g,h)\in\Sigma$,
		that is, for any bounded set $\mathfrak{D}$ of $\mathcal {P}_p (X)$,
		$$\lim\limits_{t\to\infty}\sup\limits_{(g,h)\in\Sigma}d_{\mathcal{P}(X)}\left(U^{(g,h)}\left(t,\tau\right)\mathfrak{D},\mathscr{A}\right)=0,\quad\text{for all}\,\,
		\tau\in\mathbb{R};
		$$
		\item[$(iv)$] $\mathscr{A}$ is minimal among all compact subsets of $\mathcal {P}_p (X)$ satisfying the property $(iii)$; that is, if $\mathscr{C}$ is
		any compact subset of $\mathcal {P}_p (X)$ satisfying the second property, then $\mathscr{A}\subseteq\mathscr{C}.$
	\end{itemize}
\end{definition}

\begin{definition}
	Given $(g,h)\in\Sigma$, a mapping $\Xi:\mathbb{R}\to\mathcal{P}_p(X)$ is called a complete
	solution of $U^{(g,h)}(t,\tau)$ if for every $t\in\mathbb{R}^+$ and $\tau\in\mathbb{R}$, the following relationship holds
	$$U^{(g,h)}\left(t,\tau\right)\Xi\left(\tau\right)=\Xi\left(t\right).$$
	
\end{definition}
The kernel of the process $U^{(g,h)}(t,\tau)$ is the collection $\mathcal{K}_{(g,h)}$ of all its bounded complete solutions.
The kernel section of the process $U^{(g,h)}(t,\tau)$ at time $s\in\mathbb{R}$ is given by
$$\mathcal{K}_{(g,h)}(s)=\left\{\Xi\left(s\right):\:\Xi(\cdot)\in\mathcal{K}_{(g,h)}\right\}.$$

If the family of processes $\{U^{(g,h)}(t,\tau)\}_{(g,h)\in\Sigma}$ admits a uniform measure attractor, then such an attractor has to be unique. To establish the existence of such a uniform measure attractor, we lift the family of processes to a semigroup $\{S(t)\}_{t\geq0}$ on $\mathcal{P}_p\left(X\right)\times\Sigma$ by 
$$
S(t)\left(\mu,\left(g,h\right)\right)=\left(U^{\left(g,h\right)}\left(t,0\right)\mu,T(t)\left(g,h\right)\right),\,\,\text{for every } t\geq 0, \mu \in \mathcal{P}_p\left(X\right), (g,h)\in \Sigma,
$$
which, along with Definition \ref{Def1}, can know that $S(0)=\mathrm{I}_{\mathcal{P}_{\mathrm{p}}(\mathrm{X})\times \Sigma}$ and $S(t)S(s)=S\left(t+s\right)$ for any $t\geq s\geq0$. 
Indeed, according to \cite{Chepyzhov-1994}, if the semigroup $\{S(t)\}_{t\geq0}$  possesses a global attractor in 
$\mathcal{P}_p\left(X\right)\times\Sigma$, then the family of processes $\{U^{(g,h)}(t,\tau)\}_{(g,h)\in\Sigma}$ admits a uniform measure attractor in $\mathcal{P}_p\left(X\right)$. Moreover, this attractor coincides with the projection onto $\mathcal{P}_p\left(X\right)$ of the global attractor of $\{S(t)\}_{t\geq0}$.

Building upon the theory of uniform and global attractors developed in \cite{Chepyzhov-1994,Hale-1987,Temam1997}, we state the following key abstract results, which can be also found in 
 \cite{li2024b,LZH-AML}.

\begin{theorem}\label{theo-1}
	If the semigroup $\{S(t)\}_{t\geq0}$ is continuous, point dissipative and asymptotically compact, then it has a global attractor $\mathscr{A}_{S}$ in $\mathcal{P} _p( X) \times \Sigma .$ Furthermore, if $\mathscr{A}$ is the projection of $\mathscr{A}_S$ onto $\mathcal{P}_p(X)$, 
	then $\mathscr{A}$ is the uniform measure attractors for the family of 
	processes $\{ U^{(g, h)}(t, \tau) \}_{(g, h) \in \Sigma}$. In addition, the structure of such attractors can be characterized as follows:
	$$\mathscr{A}=\bigcup_{(g,h)\in\Sigma}\mathcal{K}_{(g,h)}\left(0\right).$$
\end{theorem}

Now, the criterion of the existence and uniqueness of a uniform measure attractor shall be introduced.
\begin{theorem}\label{theo-2}
	If the family of processes $\{U^{(g,h)}(t,\tau)\}_{(g,h)\in\Sigma}$
	is jointly continuous and uniformly asymptotically compact and has a uniform closed 
	absorbing set $B$, then it has a uniform measure attractor $\mathscr{A}$. In addition, this attractor has the following structure:
	\begin{align*}%\label{llllllll}
		\mathscr{A} = \bigcup_{(g, h) \in \Sigma} \mathcal{K}_{(g,h)}(0).
	\end{align*}
\end{theorem}

\section{Well-posedness}\label{NS-Sec3}
This section shall present the existence and uniqueness of solutions to problem  \eqref{NSE-dis1.1}-\eqref{NSE-dis1.2}. For this purpose, we introduce the following notations and assumptions, which will be adopted throughout the subsequent article. 

\subsection{Basic spaces and notations}
Let $|u|$ be the modular of $u$,  $|\mathcal{O}|$ be the Lebesgue measure of domain $\mathcal{O}$, and let $C^{\infty}_{0}(\mathcal{O},\mathbb{R}^2)$ denote  the space of all infinite differentiable functions with compact support in $\mathcal{O}\subseteq \mathbb{R}^2$. Moreover, let $\|\cdot\|_{{L}^p(\mathcal{O})}$ be the norm of ${L}^p(\mathcal{O})=L^p(\mathcal{O},\mathbb{R})$ $(p\geq 1)$, let  $\mathbb{L}^p(\mathcal{O})=L^p(\mathcal{O},\mathbb{R}^2)$ with $p\geq 1$  and $\mathbb{H}^k(\mathcal{O})=H^k(\mathcal{O},\mathbb{R}^2)$ for $k\in \mathbb{Z}^+$. Let $\ell^2$ be a Hilbert space of real-valued and square-summable infinite sequences with the inner product
$$
(u,v)=\sum_{i\in \mathbb{Z}}u_iv_i, \quad  \forall u=(u_i)_{i\in \mathbb{Z}}, v=(v_i)_{i\in \mathbb{Z}} \in \ell^2,
$$
and the norm $\|u\|_{\ell^2}=\sqrt{(u, u)}$.  Let
$$
\mathcal{V}=\left\{u\in C^{\infty}_{0}(\mathcal{O},\mathbb{R}^2):\text{div~} u=0\right\},
$$
then we set

$\bullet$ $H=$ the closure of $\mathcal{V}$ in $\mathbb{L}^2(\mathcal{O})$ with inner product $(\cdot,\cdot)$ and norm $\|\cdot\|_H$;

$\bullet$ $V=$ the closure of $\mathcal{V}$  in $\mathbb{H}^1(\mathcal{O})$ with equivalent norm $\|\cdot\|_V=\|\nabla \cdot\|_{H}$;

$\bullet$ $V^*=$ the dual space of $V$ with norm $\|\cdot\|_{V^*}$.

For the relationship of $H$ and $V$, we have that the compact embedding $V\hookrightarrow H$ and the following Poincar\'{e}'s inequality
\begin{align}\label{NSE-dis3.1}
	\lambda \|u\|_{H}^2\leq \| u\|_{V}^2,\quad\forall u\in V.
\end{align}
We denote the duality product between $V$ and ${V^*}$ as $\langle \cdot,\cdot\rangle$.
Define the Stokes operator  $A: D(A)\mapsto H$ as 
$$	
Au=-\mathscr{P}\Delta u, ~~\forall u\in D(A),
$$
where the domain $D(A)=\mathbb{H}^2(\mathcal{O})\cap V$. Let $\mathscr{P}$ denote the Leray projection from $\mathbb{L}^2(\mathcal{O})$ to ${H}$.

Furthermore, for any $u,v\in V$, we define the bilinear map $B(\cdot,\cdot):V\times V\rightarrow V^*$ as follows:
$$
B(u,v)=\mathscr{P}(u\cdot \nabla )v,
$$
and the trilinear form $b(\cdot,\cdot,\cdot):V\times V \times V\rightarrow \mathbb{R}$ is defined by
\begin{align*}
	b(u,v,w)=\sum_{i,j=1}^2 \int_{\mathcal{O}} u_i(x)\frac{\partial v_j(x)}{\partial x_i}w_j(x)dx, \quad \forall u,v,w\in V.
\end{align*}
Define the  relation between bilinear and trilinear operators as follows:
$$
\langle B(u,v), w\rangle =b(u,v,w), \quad\forall u,v, w\in V,
$$
then using the integration by parts can obtain
\begin{align}\label{NSE-dis3.--002}
	b(u,v,w)=-b(u,w,v),\quad b(u,v,v)=0, \quad \forall u,v,w\in V.
\end{align}
In particular, for any $u,v,w\in V$, it holds
\begin{align}\label{NSE-dis3.2}
	\langle B(u,v), w\rangle \leq \mathfrak{C}\|\nabla v\|_{H} \|u\|_{H}^{1/2} \|\nabla u\|_{H}^{1/2} \|w\|_{H}^{1/2} \|\nabla w\|_{H}^{1/2}.
%	\langle B(u,v), w\rangle \leq \mathfrak{C}|w|_2 |u|_2^{1/2} |\nabla u|_2^{1/2} |\nabla v|_2^{1/2} |A v|_2^{1/2}.
\end{align}

\subsection{Assumptions and Leray projection model} 
Throughout the entire paper, let $\delta_0$ be
the Dirac probability measure at $0$, we denote by $g_0(t),h_0(t)\in C_b(\mathbb{R},H)$ two almost periodic functions,
and assume that the function
$f:   \mathcal{O}
\times {\mathbb{R}}^2
\times \mathcal{P}_2(H)
\to {\mathbb{R}^2}$ fulfills the following  conditions:

$(H_1)$ For all $x \in \mathcal{O}$, $u, u_1, u_2  \in \mathbb{R}^2$ 
and $\mu, \mu_1, \mu_2 \in \mathcal{P}_2(H)$,
\begin{align}
	f(x, 0, \delta_0)&=0,\label{NSE-dis-f_1}\\
	|f(x,u,\mu)| \leq \phi_1(x)(1+&|u|)+\psi_1 (x) \sqrt{\mu(\|\cdot\|_{H}^2)},\label{NSE-dis-f_2}\\
		| f(x,u_1, \mu_1) - f(x,u_2, \mu_2) |
		\leq\, &\phi_2(x) |u_1 - u_2|+ \psi_2(x) \mathbb{W}_2(\mu_1, \mu_2),\label{NSE-dis-f_3}
\end{align}
where $\phi_i \in
L^\infty(\mathcal{O})$ and $\psi_i\in L^2(\mathcal{O})$ $(i = 1,2)$.

$(H_2)$ The function $h(t,x)$ is almost periodic and additionally satisfies
\begin{align}\label{NSE-dis-gh_03}
	\int_t^{t+1} \|\nabla h(s)\|_H^2ds \leq {\bf \hat{c}}<\infty.
\end{align}
Moreover, the time-dependent external forcing terms $g,h:\mathbb{R}\rightarrow H$ satisfy
\begin{align}\label{NSE-dis-gh_3}
	(g,h)\in \Sigma.
\end{align}
%For all $x \in \mathcal{O}$ and $t\in \mathbb{R}$, 
%% $g,h:\mathbb{R}\rightarrow H$, as well as $g(t)=g(t,x)$ and 
%For all $x \in \mathcal{O}$, the time-dependent external terms taken values in $H$ $g(t), h(t)$ , i.e., 
%For all $x \in \mathcal{O}$ and $t\in \mathbb{R}$, let $\|h(t)\|^2=\sum_{k=1}^\infty \|h(t)\|_H^2$
%Assume that the function  $g: \mathbb{R}\times \mathcal{O}\rightarrow \mathbb{R}^2$ satisfies $g\in L^2_{loc}(\mathbb{R}; H)$.

$(H_3)$ For each $k\in \mathbb{N}$, $\sigma_k:\mathbb{R}\times\mathbb{R}^2\times \mathcal{P}_2(H)\rightarrow \mathbb{R}^2$ is continuous such that for all $t\in \mathbb{R}$, $u\in \mathbb{R}^2$ and $\mu \in \mathcal {P}_2(H)$,
\begin{align}\label{NSE-dis-G1}
	|\sigma_k(t,u,\mu)|\leq \beta_k\left(1+\sqrt{ \mu (\|\cdot \|_H^2 )}\right)+\hat{\gamma}_k|u|,
\end{align}
where $\beta=\{\beta_{k}\}_{k=1}^\infty$ and $\hat{\gamma}=\{\hat{\gamma}_k\}_{k=1}^\infty$
are  the nonnegative sequences with $\|\beta\|_{\ell^2}^2+\|\hat{\gamma}\|_{\ell^2}^2=\sum_{k=1}^\infty
\left(\beta_k^2 + \hat{\gamma}_k^2\right)< \infty$.	

Moreover, we assume that $\sigma_k(t,u,\mu)$ is differentiable with respect to $u$ and uniformly Lipschitz continuous in both $u$ and $\mu$. That is, for every $k\in \mathbb{N}$, there exists a constant $L_k>0$ such that for all $t\in \mathbb{R}$,  $u_1, u_2 \in \mathbb{R}^2$ and $\mu_1, \mu_2 \in \mathcal {P}_2(H)$, it holds
\begin{align}\label{NSE-dis-G2}
	|\sigma_k(t,u_1, \mu_1) - \sigma_k(t,u_2, \mu_2) |
	\leq L_{k} \left( | u_1 - u_2 | + \mathbb{W}_2 (\mu_1, \mu_2) \right),
\end{align}
where   $L=\{L_{k}\}_{k=1}^\infty$ is a nonnegative sequence
such that  $\|L\|_{\ell^2}^2= \sum_{k=1}^\infty L_{k}^2< \infty$.

We can infer from \eqref{NSE-dis-G2} that for all $t\in \mathbb{R}$, $u \in \mathbb{R}^2$ and $\mu \in \mathcal {P}_2(H)$,
\begin{align}\label{NSE-dis-G3}
	\left|{\frac {\partial \sigma_k}{\partial u}}(t,u , \mu )\right|
	\leq L_{k}.
\end{align}

For each $t\in \mathbb{R}$, $u \in H$ and $\mu \in \mathcal {P}_2(H)$, we define a map
$\sigma(t,u,\mu): l^2 \rightarrow H$ by
$$
\sigma(t,u, \mu)(\zeta)(x)=\sum_{k=1}^\infty \left(h(t,x)+\kappa(x)\sigma_k(t,u,\mu)\right)\zeta_k,\quad \forall \zeta=\{\zeta_k\}_{k=1}^\infty\in \ell^2, x\in\mathcal{O}.
$$
%For two separable Hilbert spaces $l^2$ and $H$,
Denote by $\mathcal{L}_2(\ell^2,H)$ the space of all Hilbert-Schmidt operators from $l^2$ to $H$, it is endowed with norm  $\|\cdot\|_{\mathcal{L}_2(l^2;H)}$. Then, we can deduce from \eqref{NSE-dis-gh_3} and \eqref{NSE-dis-G1} that for any $u\in H$ and $\mu \in \mathcal {P}_2(H)$,
\begin{align}\label{NSE-dis-G4}
	\begin{split}
		&~~\|\sigma(t,u, \mu)\|_{\mathcal{L}_2(l^2;H)}^2
		=\sum_{k=1}^\infty\int_{\mathcal{O}} |h(t,x)+\kappa(x)\sigma_k(t,u,\mu)|^2dx\\
		%&\leq 2\|h(t)\|_H^2+8\|\beta\|_{\ell^2}^2\|\kappa\|^2_{L^\infty(\mathcal{O})}|\mathcal{O}|\left(1+\mu(\|\cdot\|_H^2)\right) +4\|\kappa\|^2_{L^\infty(\mathcal{O})}\|\gamma\|_{\ell^2}^2\|u\|_{H}^2\\
		&\leq 2\|h_0\|_{C_b(\mathbb{R},H)}^2+8\|\beta\|_{\ell^2}^2\|\kappa\|^2_{L^\infty(\mathcal{O})}|\mathcal{O}|\left(1+\mu(\|\cdot\|_H^2)\right) +4\|\kappa\|^2_{L^\infty(\mathcal{O})}\|\hat{\gamma}\|_{\ell^2}^2\|u\|_{H}^2<\infty,
	\end{split}
\end{align}
where we used $\|h(t)\|_H^2\leq \|h_0\|_{C_b(\mathbb{R},H)}^2$.
Furthermore, by \eqref{NSE-dis-G2} we find that for all $u_1,u_2 \in H$ and $\mu_{1},\mu_{2}\in \mathcal {P}_2(H)$,
\begin{align}\label{NSE-dis-G5}
	\begin{split}
	&\|\sigma(t,u_1, \mu_1)-\sigma(t,u_2, \mu_2)\|_{\mathcal{L}_2(l^2;H)}^2\\
	=&\sum_{k=1}^\infty\int_{\mathcal{O}}|\kappa(x)|^2|\sigma_k(t,u_1, \mu_1) - \sigma_k(t,u_2, \mu_2) |^2dx\\
	\leq& 2\|\kappa\|^2_{L^\infty(\mathcal{O})}\|L\|_{\ell^2}^2(1+|\mathcal{O}|)\left(\|u_1-u_2\|_H^2+\mathbb{W}_2^2 (\mu_1, \mu_2)\right).
	\end{split}
\end{align}

%$(H_4)$ 
 Under the framework of hypothesis $(H_2)$, the time-dependent function $h$ is required to satisfy Conditions \eqref{NSE-dis-gh_03} and \eqref{NSE-dis-gh_3}. To illustrate this constructively, we now present a concrete exemplification.
 
 \begin{example}
 We consider the following separated variable form:
 $$
 h(t,x)=h_1(t)\cdot h_2(x), ~~\forall x\in \mathcal{O}, t\in \mathbb{R},
 $$
 where $h_1(t)$ is an almost periodic function with respect to $t$, and $h_2(x)$ is a smooth function with bounded derivatives. Then, $h(t,x)$ is almost periodic and $|\nabla h|$ is bounded. The specific forms of $h_1(t)$ and $h_2(x)$ may be assumed as follows:
 
 Define the functions $h_1(t)=\sin t+ \sin{\sqrt{2}t}$ and $h_2(x)=\arctan{x}$, and consider 
 $$
 h(t,x)=h_1(t)\cdot h_2(x)=\left(\sin t+ \sin{\sqrt{2}t}\right)\cdot\arctan{x}.
 $$
 It is easy to see that $h(t, x)$ is almost periodic with respect to $t$. Moreover, since $|\nabla h(t,x)|=|h_1(t) \cdot \frac{1}{1+x^2}|$, we have
 \begin{align*}
 	\int_t^{t+1} \|\nabla h(s)\|_H^2ds \leq |\mathcal{O}|\int_t^{t+1} |h_1(s)|^2ds\leq 4|\mathcal{O}|,
 \end{align*}
 then the integral $\int_t^{t+1} \|\nabla h(s)\|_H^2ds$ is uniformly bounded. 
 \end{example}
%{\bf Example:} it is easy to verify that

With the help of the aforementioned notations, we can rewrite \eqref{NSE-dis1.1}  into the following Leray projection form.
\begin{align}\label{NSE-dis1.-01}
	\left\{
	\begin{aligned}
		&du(t)+\nu Au(t)dt+B(u(t),u(t))dt\\
		&\qquad = g(t)dt +f(x, u (t), \mathscr{L}_{u(t)} ) dt+\varepsilon \sigma\left(t,u(t),\mathscr{L}_{u(t)}\right)d W(t),\\
		&u(x,\tau)=u_\tau(x), \quad \text{for}\,\,x\in \mathcal{O}.
	\end{aligned}
	\right.
\end{align}

\subsection{Existence and uniqueness of solutions}
In this subsection, we shall present the well-posedness result of problem \eqref{NSE-dis1.-01} under the following definition of the existence of solutions.
\begin{definition}\label{NSE-dis-def2.1} 
	Let  $(\Omega, \mathscr{F}, \{\mathscr{F}_t\}_{t\in \mathbb{R}}, \mathbb{P})$ be a fixed complete filtered probability space,
	for every $\tau \in \mathbb{R}$, $T>0$, $p\geq 2$, $\varepsilon \in (0,1]$, $u_\tau \in L^p(\Omega, \mathscr{F}_\tau; H)$,
	a continuous $H$-valued $\mathscr{F}_t$-adapted stochastic process $u$ is called a strong
	solution of problem \eqref{NSE-dis1.-01} if
	$$
	u \in C([\tau,\tau+T]; H)\cap L^2((\tau,\tau+T); V)~~\mathbb{P}\text{-almost surely},
	$$
	and for each $t\in [\tau,\tau+T] $ and $v \in V$, the following equality holds, $\mathbb{P}$--almost surely,
	\begin{align*}
		&\,\, (u (t), v)
		+ \nu\int_\tau^t
		\langle Au(s),v\rangle ds
		+ \int_\tau^t  \langle B(u(s),u(s)), v \rangle ds -\int_\tau^t \left(f(x, u (s), \mathscr{L}_{u (s)} ),v\right) ds
		\nonumber \\
		&=  (u_\tau, v)
		+ \int_\tau^t \left(g(s),v\right)ds
		+ {\varepsilon} \int_\tau^t
		\left(\sigma(s,u(s), \mathscr{L}_{u(s)} ) dW(s), v
		\right).
	\end{align*}
\end{definition}%a.s.

We now formulate the existence and uniqueness of solutions of problem \eqref{NSE-dis1.-01}, see e.g., \cite[Theorem $3.2$]{QZH-2D-2025}.
\begin{theorem}\label{NSE-dis-well2.2}
	Suppose that $(H_1)-(H_3)$ hold.
	Then, for every $\tau \in \mathbb{R}$, $T>0$, $\varepsilon\in (0,1]$, $p\geq 2$, $u_\tau \in L^p(\Omega, \mathscr{F}_\tau; H)$,
	the problem \eqref{NSE-dis1.-01} has a unique solution
	$u$	under the sense of Definition \ref{NSE-dis-def2.1}, and it satisfies the energy equality, for all $t\in [\tau,\tau+T]$,
	\begin{align}\label{NSE-dis-def2.2}
		\begin{split}
			& \| u (t) \|_H^2 + 2\nu \int_\tau^t\|u(s)\|_V^2 ds-2 \int_\tau^t (f(x, u (s), \mathscr{L}_{u(s)} ),  u (s) ) ds \\
			& =  \| u_\tau \|_H^2 +2 \int_\tau^t (g(s),u(s))ds
			+ \varepsilon^2 \int_\tau^t \| \sigma(s,u (s), \mathscr{L}_{u (s)} ) \|_{\mathcal{L}_2(\ell^2,H)}^2 ds  \\
			&\quad + 2 \varepsilon \int_\tau^t
			\left( u (s), \sigma(s,u (s), \mathcal{L}_{u (s)}) dW(s)
			\right), \ \ \    \mathbb{P} \text{--almost surely}
		\end{split}
	\end{align}
	Moreover, the following uniform estimates are valid:
	\begin{align*}
		\mathbb{E}\left[\sup_{t\in [\tau,\tau+T]}\| u (t)\|_{H}^{p}\right]
		+ \int_{\tau}^{\tau+T} \mathbb{E}\left[\| u(s) \|_{H}^{p-2}\| u(s) \|_{V}^{2}\right]ds
		\leq  C
		\left( 1 + \mathbb{E}\left[\| u_\tau \|_{H}^{p}\right]
		\right),
	\end{align*}
	where $C=C(\tau,T) > 0$ is a constant independent of $u_\tau$ and $\varepsilon$.
\end{theorem}

The result of the existence and uniqueness of weak solutions for problem \eqref{NSE-dis1.-01} is given below.
\begin{proposition}\label{WE-unds1the}
	If $(H_1)-(H_3)$ hold, then the problem \eqref{NSE-dis1.-01} has a unique weak solution for every initial distribution in $\mathcal{P}_p(H)$ with $p\geq 2$. Weak uniqueness of problem \eqref{NSE-dis1.-01} in $\mathcal{P}_p(H)$ means that for any $\tau \in \mathbb{R}$, any two weak solutions of 	\eqref{NSE-dis1.-01} from $\tau$ with a common initial distribution in $\mathcal{P}_p(H)$ are equal in law; more precisely, if $(u(t),W(t))_{t\geq \tau}$ with respect to $(\Omega, \mathscr{F}, \{\mathscr{F}_t\}_{t\geq \tau}, \mathbb{P})$ and $(\widetilde{u}(t),\widetilde{W}(t))_{t\geq \tau}$ with respect to $(\widetilde{\Omega}, \widetilde{\mathscr{F}}, \{\widetilde{\mathscr{F}}_t\}_{t\geq \tau}, \widetilde{\mathbb{P}})$ are weak solutions of \eqref{NSE-dis1.-01} with the same initial distribution $\mathscr{L}_{u_\tau}|_{\mathbb{P}}=\mathscr{L}_{\widetilde{u}_\tau}|_{\widetilde{\mathbb{P}}}$, then $\mathscr{L}_{u(\cdot)}|_{\mathbb{P}}=\mathscr{L}_{\widetilde{u}(\cdot)}|_{\widetilde{\mathbb{P}}}$.
\end{proposition}
\begin{proof}
	The proof of this proposition is similar to those in \cite[Theorem 2.1]{WangFy-2018} and \cite[Proposition 3.2]{SSLd-2024}. We only provide a brief outline of the steps. The proof of the existence of weak solutions is similar to that in \cite{QZH-2D-2025}, and we present the proof of weak uniqueness. Let $(u(t),W(t))_{t\geq \tau}$  and $(\widetilde{u}(t),\widetilde{W}(t))_{t\geq \tau}$ respect to $(\Omega, \mathscr{F}, \{\mathscr{F}_t\}_{t\geq \tau}, \mathbb{P})$ and $(\widetilde{\Omega}, \widetilde{\mathscr{F}}, \{\widetilde{\mathscr{F}}_t\}_{t\geq \tau}, \widetilde{\mathbb{P}})$, respectively,
	be two weak solutions of \eqref{NSE-dis1.-01} such that $\mathscr{L}_{u_\tau}|_{\mathbb{P}}=\mathscr{L}_{\widetilde{u}_\tau}|_{\widetilde{\mathbb{P}}}$. Then, $u(t)$ solves \eqref{NSE-dis1.-01} while $\widetilde{u}(t)$ solves the following equation:
	\begin{align}\label{WE-unds111}
	\begin{split}
			&\,\, d\widetilde{u}(t)+\nu A\widetilde{u}(t)dt+B(\widetilde{u}(t),\widetilde{u}(t))dt\\
			& = g(t)dt +f(x, \widetilde{u}(t), \mathscr{L}_{\widetilde{u}(t)}|_{\widetilde{\mathbb{P}}} ) dt+\varepsilon \sigma\left(t,\widetilde{u}(t),\mathscr{L}_{\widetilde{u}(t)}|_{\widetilde{\mathbb{P}}}\right)d \widetilde{W}(t).
	\end{split}
	\end{align}
	In order to prove $\mathscr{L}_{u(\cdot)}|_{\mathbb{P}}=\mathscr{L}_{\widetilde{u}(\cdot)}|_{\widetilde{\mathbb{P}}}$, let $\mathscr{L}_{u(t)}|_{\mathbb{P}}=\mu_t$, and consider the stochastic differential equation:
	 	\begin{align}\label{WE-unds112}
	 	d\bar{u}(t)+\nu A\bar{u}(t)dt+B(\bar{u}(t),\bar{u}(t))dt = g(t)dt +f(x, \bar{u}(t),\mu_t ) dt+\varepsilon \sigma\left(t,\bar{u}(t),\mu_t\right)d \widetilde{W}(t)
	 \end{align}
	 with $\bar{u}_\tau=\widetilde{u}_{\tau}$. By  $(H_1)-(H_3)$, we know that \eqref{WE-unds112} has a unique strong solution. It follows from the Yamada-Watanabe theorem that the weak solution of \eqref{WE-unds112}  is also unique. Note that
	 \begin{align*}%\label{WE-unds113}
%	 	\begin{split}
	 		&\,\, du(t)+\nu Au(t)dt+B(u(t),u(t))dt\\
	 		& = g(t)dt +f(x, u(t),\mu_t ) dt+\varepsilon \sigma\left(t,u(t),\mu_t\right)d {W}(t), \quad  \mathscr{L}_{u_\tau}|_{\mathbb{P}}=\mathscr{L}_{\widetilde{u}_\tau}|_{\widetilde{\mathbb{P}}}.
%	 	\end{split}
	 \end{align*}
	 Consequently, by the weak uniqueness of \eqref{WE-unds112}, it follows that $\mathscr{L}_{u(\cdot)}|_{\mathbb{P}}=\mathscr{L}_{\bar{u}(\cdot)}|_{\widetilde{\mathbb{P}}}$. 
	 
	 With this in mind, we rewrite \eqref{WE-unds112} as
	 \begin{align*}
	 	d\bar{u}(t)+\nu A\bar{u}(t)dt+B(\bar{u}(t),\bar{u}(t))dt = g(t)dt +f(x, \bar{u}(t),\mathscr{L}_{\bar{u}(\cdot)}|_{\widetilde{\mathbb{P}}}) dt+\varepsilon \sigma\left(t,\bar{u}(t),\mathscr{L}_{\bar{u}(\cdot)}|_{\widetilde{\mathbb{P}}}\right)d \widetilde{W}(t)
	 \end{align*}
	 with $\bar{u}_\tau=\widetilde{u}_{\tau}$.
	 Since the equation \eqref{WE-unds111} has a unique strong solution, we obtain $\widetilde{u}=\bar{u}$. Therefore, it holds $\mathscr{L}_{u(\cdot)}|_{\mathbb{P}}=\mathscr{L}_{\bar{u}(\cdot)}|_{\widetilde{\mathbb{P}}}=\mathscr{L}_{\widetilde{u}(\cdot)}|_{\widetilde{\mathbb{P}}}$. This complete the proof.
\end{proof}

\section{Long-time uniform estimates of solutions}\label{NS-Sec4}

In this section, we derive some uniform a priori estimates for the solution of problem \eqref{NSE-dis1.-01}, which serve as the foundation for proving the existence and uniqueness of pullback measure attractors.
To this end, let 
\begin{align*}
	\mathbf{k}_0&=6\|\phi_1\|_{L^\infty(\mathcal{O})}+6\|\psi_1\|_{L^2(\mathcal{O})}+6\|\beta\|_{\ell^2}^2+8\|\kappa\|^2_{L^\infty(\mathcal{O})}\|\hat{\gamma}\|_{\ell^2}^2+16\|\beta\|_{\ell^2}^2\|\kappa\|^2_{L^\infty(\mathcal{O})}|\mathcal{O}|,
\end{align*}
we assume that
\begin{align}\label{NSE-dis-def4.1}
	\nu>\frac{2\mathbf{k}_0}{\lambda }.
\end{align}
From \eqref{NSE-dis-def4.1}, it is straightforward to deduce that there exists a sufficiently small constant $\gamma\in \left(0,\frac{1}{2}\right)$ such that
\begin{align}\label{NSE-dis-def4.1*}
	\frac{\nu\lambda }{2}-2\gamma>\mathbf{k}_0.
\end{align}

%...........................................................................

\begin{lemma}\label{NSE-dis-lemma4.1*}
	Suppose that $(H_1)-(H_3)$, \eqref{NSE-dis-def4.1} and \eqref{NSE-dis-def4.1*} hold. Then, for any $R>0$, there exists $T=T(R)>0$ such that for all $\tau \in \mathbb{R}$, $t-\tau\geq T$ and $\varepsilon\in (0,1]$, the solution $u$ of problem \eqref{NSE-dis1.-01} satisfies
	\begin{align*}
		\mathbb{E}\left[\|u(t,\tau,u_{\tau})\|_H^2\right]+\int_{\tau}^t e^{\gamma(s-t)}\mathbb{E} \left[\| u(s,\tau,u_{\tau})\|_V^2\right]ds\leq \mathcal{M}_1,
	\end{align*}
	where $u_{\tau}\in L^2(\Omega,\mathscr{F}_{\tau};H)$ with $\mathbb{E}\left[\|u_\tau\|_H^2\right]\leq R$, and $\mathcal{M}_1>0$ is a constant that depends on $\gamma, \nu, |\mathcal{O}|, \|\phi_1\|_{L^\infty(\mathcal{O})},\|\beta\|_{\ell^2},\|\kappa\|_{L^{\infty}(\mathcal{O})}, \|\hat{\gamma}\|_{\ell^2}, g_0, h_0$, but does not depend on $\varepsilon, \tau, u_{\tau}$ and $(g,h)\in \Sigma$. 
\end{lemma}

\begin{proof} Applying It\^{o}'s formula to the process $\|u(t)\|_H^p$ $(p\geq 2)$ and using \eqref{NSE-dis1.-01}, we can obtain that for all $t\geq \tau$,
	\begin{align}\label{3DSTNSeq4.3}
		\begin{split}
			&\,\, \|u(t)\|_H^p+\frac{\nu p}{4}\int_\tau^t \|u(s)\|_H^{p-2}\|u(s)\|_V^2ds+\frac{3\nu \lambda p}{4}\int_\tau^t \|u(s)\|_H^{p}ds\\
			&\leq \|u_\tau\|_H^p+p\int_\tau^t \|u(s)\|_H^{p-2}\left(g(s),u(s)\right)ds +
			p\int_\tau^t \|u(s)\|_H^{p-2}\left(f(u(s),\mathscr{L}_{u(s)},u(s))\right)ds \\
			& + \varepsilon p\int_\tau^t \|u(s)\|_H^{p-2} \left(u(s),\sigma(s,u(s),\mathscr{L}_{u(s)})\right)dW(s)\\
			&+\frac{p(p-1)}{2}\varepsilon ^2\int_\tau^t \|u(s)\|_H^{p-2}\|\sigma(s,u(s),\mathscr{L}_{u(s)})\|^2_{\mathcal{L}_2(\ell^2,H)}ds,
		\end{split}
	\end{align}
	$\mathbb{P}$-almost surely. For each  $m\in \mathbb{N}$, we define a stopping time
	$\tau_m$ given by
	$$
	\tau_m=\inf\{t\geq \tau:  \| u(t) \|_H > m\}.
	$$
	By convention, $\inf \emptyset =+\infty$. 
	From \eqref{3DSTNSeq4.3}, we can infer that for all $t\geq \tau$, it holds
	\begin{align}\label{3DSTNSeq4.5}
		\begin{split}
			&\,\, \mathbb{E}\left[e^{\gamma (t\wedge \tau_m)}\|u(t\wedge \tau_m)\|_H^2\right]+\frac{\nu }{2}\mathbb{E}\left[\int_\tau^{t\wedge \tau_m} e^{\gamma s}\|u(s)\|_V^2ds\right]\\
			&+\left(\nu\lambda-\gamma\right)\mathbb{E}\left[\int_\tau^{t\wedge \tau_m} e^{\gamma s}\|u(s)\|_H^2ds\right]\\
			&\leq e^{\gamma\tau}\mathbb{E}\left[\|u_\tau\|_H^2\right]+\frac{2}{\nu\lambda \gamma}\|g_0\|_{C_b(\mathbb{R},H)}^{2}e^{\gamma t}+
			2\mathbb{E}\left[\int_\tau^{t\wedge \tau_m} e^{\gamma s}\left(f(u(s),\mathscr{L}_{u(s)},u(s))\right)ds\right] \\
			&+\varepsilon ^2\mathbb{E}\left[\int_\tau^{t\wedge \tau_m} e^{\gamma s}\|\sigma(s,u(s),\mathscr{L}_{u(s)})\|^2_{\mathcal{L}_2(\ell^2,H)}ds\right],
		\end{split}
	\end{align}
	where we used the following inequality
	\begin{align*}
		2\int_\tau^{t\wedge \tau_m} e^{\gamma s}\left(g(s),u(s)\right)ds
		\leq& \frac{\nu\lambda  }{2}\int_\tau^{t\wedge \tau_m} e^{\gamma s}\|u(s)\|_H^2ds+\frac{2}{\nu\lambda}\int_\tau^{t\wedge \tau_m} e^{\gamma s}\|g(s)\|_H^{2}ds\\
		\leq& \frac{\nu\lambda  }{2}\int_\tau^{t\wedge \tau_m} e^{\gamma s}\|u(s)\|_H^2ds+\frac{2}{\nu\lambda \gamma}\|g_0\|_{C_b(\mathbb{R},H)}^{2}e^{\gamma t}.
	\end{align*}
	
	In what follows, we estimate the third and fourth terms on the right-hand side of \eqref{3DSTNSeq4.5}. For the third term on the right-hand side of \eqref{3DSTNSeq4.5}, by H\"{o}lder's inequality, Young's inequality and \eqref{NSE-dis-f_2} in $(H_1)$ we have
	\begin{align}\label{3DSTNSeq4.6}
%		\begin{split}
			&\,\,2\mathbb{E}\left[\int_\tau^{t\wedge \tau_m} e^{\gamma s}\left(f(\cdot,u(s),\mathscr{L}_{u(s)}),u(s)\right)ds\right] \nonumber \\
			&\leq 2\mathbb{E}\left[\int_\tau^{t\wedge \tau_m} e^{\gamma s}\int_{\mathcal{O}} \left(|\phi_1(x)||u(s)|+|\phi_1(x)||u(s)|^2+|\psi_1(x)|\sqrt{\mathbb{E}\left[\|u(s)\|_H^2\right]}|u(s)|\right)dxds\right] \nonumber \\
			&\leq \frac{4}{\nu \lambda}\int_\tau^t e^{\gamma s}\|\phi_1\|_{L^2(\mathcal{O})}^2ds+\|\psi_1\|_{L^2(\mathcal{O})}\mathbb{E}\left[\int_\tau^{t\wedge \tau_m} e^{\gamma s}\mathbb{E}\left[\|u(s)\|_H^2\right]ds\right]\nonumber \\
			&+ \mathbb{E}\left[\int_\tau^{t\wedge \tau_m} e^{\gamma s}\left(\frac{\nu\lambda}{4}+2\|\phi_1\|_{L^\infty(\mathcal{O})}+\|\psi_1\|_{L^2(\mathcal{O})}\right)\|u(s)\|_H^2ds\right]\nonumber \\
			&\leq \frac{4|\mathcal{O}|}{\nu \lambda\gamma}\|\phi_1\|_{L^\infty(\mathcal{O})}^2e^{\gamma t}+\frac{\nu\lambda}{4}\mathbb{E}\left[\int_\tau^{t\wedge \tau_m} e^{\gamma s}\|u(s)\|_H^2ds\right]\nonumber \\
			&+ 2\left(\|\phi_1\|_{L^\infty(\mathcal{O})}+\|\psi_1\|_{L^2(\mathcal{O})}\right)\int_\tau^{t} e^{\gamma s}\mathbb{E}\left[\|u(s)\|_H^2\right]ds.
%		\end{split}
	\end{align}
%	where $c_1=\sqrt{|\mathcal{O}|}$ denotes the embedding constant for $L^\infty(\mathcal{O})\hookrightarrow H$ and $|\mathcal{O}|$ represents the Lebesgue measure of domain $\mathcal{O}$.
	For  the fourth term on the right-hand side of \eqref{3DSTNSeq4.5}, by \eqref{NSE-dis-G4} we have 
	\begin{align}\label{3DSTNSeq4.7}
		\begin{split}
			&\,\,\varepsilon ^2\mathbb{E}\left[\int_\tau^{t\wedge \tau_m} e^{\gamma s}\|\sigma(s,u(s),\mathscr{L}_{u(s)})\|^2_{\mathcal{L}_2(\ell^2,H)}ds\right]\\
			&\leq \frac{2}{\gamma}\|h_0\|_{C_b(\mathbb{R},H)}^2e^{\gamma t}+
			8\|\beta\|_{\ell^2}^2\|\kappa\|^2_{L^{\infty}(\mathcal{O})}|\mathcal{O}|\mathbb{E}\left[\int_\tau^t e^{\gamma s}\left(1+\mathbb{E}\left[\|u(s)\|_H^2\right]\right)ds\right]\\
			&+4\|\kappa\|^2_{L^{\infty}(\mathcal{O})}\|\hat{\gamma}\|_{\ell^2}^2\mathbb{E}\left[\int_\tau^{t\wedge \tau_m} e^{\gamma s}\|u(s)\|_H^2 ds\right]\\
			&\leq \frac{2}{\gamma}\|h_0\|_{C_b(\mathbb{R},H)}^2e^{\gamma t}+
			\frac{8|\mathcal{O}|}{\gamma}\|\beta\|_{\ell^2}^2\|\kappa\|^2_{L^{\infty}(\mathcal{O})} e^{\gamma t}\\
			&+4\|\kappa\|^2_{L^{\infty}(\mathcal{O})}\left(2\|\beta\|_{\ell^2}^2|\mathcal{O}|+\|\hat{\gamma}\|_{\ell^2}^2\right)\int_\tau^t e^{\gamma s}\mathbb{E}\left[\|u(s)\|_H^2\right]ds.
		\end{split}
	\end{align} 
	
	Together with \eqref{3DSTNSeq4.5}-\eqref{3DSTNSeq4.7}, for all $t\geq \tau$, it holds
	\begin{align}\label{3DSTNSeq4.8}
		\begin{split}
			&\,\, \mathbb{E}\left[e^{\gamma (t\wedge \tau_m)}\|u(t\wedge \tau_m)\|_H^2\right]+\frac{\nu }{2}\mathbb{E}\left[\int_\tau^{t\wedge \tau_m} e^{\gamma s}\|u(s)\|_V^2ds\right]\\
			&+\left(\frac{3\nu\lambda }{4}-\gamma\right)\mathbb{E}\left[\int_\tau^{t\wedge \tau_m} e^{\gamma s}\|u(s)\|_H^2ds\right]\\
			&\leq e^{\gamma\tau}\mathbb{E}\left[\|u_\tau\|_H^2\right]+\frac{2}{\gamma}\left(\frac{1}{\nu\lambda}\|g_0\|_{C_b(\mathbb{R},H)}^{2}+\|h_0\|_{C_b(\mathbb{R},H)}^2\right)e^{\gamma t}\\
			&+\underbrace{2\left(\|\phi_1\|_{L^\infty(\mathcal{O})}+\|\psi_1\|_{L^2(\mathcal{O})}+2\|\kappa\|^2_{L^{\infty}(\mathcal{O})}\left(2\|\beta\|_{\ell^2}^2|\mathcal{O}|+\|\hat{\gamma}\|_{\ell^2}^2\right)\right)}_{{\bf k}_1}\int_\tau^{t} e^{\gamma s}\mathbb{E}\left[\|u(s)\|_H^2\right]ds\\
			&+\underbrace{\frac{4|\mathcal{O}|}{\gamma}\left(\frac{1}{\nu \lambda}\|\phi_1\|_{L^\infty(\mathcal{O})}^2+
				2\|\beta\|_{\ell^2}^2\|\kappa\|^2_{L^{\infty}(\mathcal{O})}\right)}_{{\bf k}_2} e^{\gamma t}.
		\end{split}
	\end{align} 
	By passing to the limit as $m \to \infty$ in \eqref{3DSTNSeq4.8} and applying Fatou's lemma, we can obtain that
	for all $t\geq \tau$,
	\begin{align*}%\label{3DSTNSeq4.9}
		&\,\, e^{\gamma t}\mathbb{E}\left[\|u(t)\|_H^2\right]+\frac{\nu }{2}\int_\tau^{t} e^{\gamma s}\mathbb{E}\left[\|u(s)\|_V^2\right]ds+\left(\frac{\nu\lambda }{4}+\gamma\right)\int_\tau^t e^{\gamma s}\mathbb{E}\left[\|u(s)\|_H^2\right]ds\\
		&\leq e^{\gamma\tau}\mathbb{E}\left[\|u_\tau\|_H^2\right]+\frac{2}{\gamma}\left(\frac{1}{\nu\lambda}\|g_0\|_{C_b(\mathbb{R},H)}^{2}+\|h_0\|_{C_b(\mathbb{R},H)}^2\right)e^{\gamma t}\\
		&+\left(2\gamma-\frac{\nu\lambda }{2}+{\bf k}_1\right)\int_\tau^{t} e^{\gamma s}\mathbb{E}\left[\|u(s)\|_H^2\right]ds+{\bf k}_2e^{\gamma t}\\
		&\leq e^{\gamma\tau}\mathbb{E}\left[\|u_\tau\|_H^2\right]+\frac{2}{\gamma}\left(\frac{1}{\nu\lambda}\|g_0\|_{C_b(\mathbb{R},H)}^{2}+\|h_0\|_{C_b(\mathbb{R},H)}^2\right)e^{\gamma t}\\
		&+\left(2\gamma-\frac{\nu\lambda }{2}+{\bf k}_0\right)\int_\tau^{t} e^{\gamma s}\mathbb{E}\left[\|u(s)\|_H^2\right]ds+{\bf k}_2e^{\gamma t},
	\end{align*} 
	which, together with \eqref{NSE-dis-def4.1}, yields that for all $t\geq \tau$,
	\begin{align}\label{3DSTNSeq4.9}
		\begin{split}
			&\,\, \mathbb{E}\left[\|u(t)\|_H^2\right]+\frac{\nu }{2}\int_\tau^{t} e^{\gamma (s-t)}\mathbb{E}\left[\|u(s)\|_V^2\right]ds+\gamma\int_\tau^t e^{\gamma (s-t)}\mathbb{E}\left[\|u(s)\|_H^2\right]ds\\
			&\leq e^{\gamma(\tau-t)}\mathbb{E}\left[\|u_\tau\|_H^2\right]+\frac{2}{\gamma}\left(\frac{1}{\nu\lambda}\|g_0\|_{C_b(\mathbb{R},H)}^{2}+\|h_0\|_{C_b(\mathbb{R},H)}^2\right)
			+{\bf k}_2.
		\end{split}
	\end{align}

	Thanks to $\mathbb{E}\left[\|u_\tau\|_H^2\right]\leq R$, we have
	$$
	\lim_{t\rightarrow \infty}e^{-\gamma(t-\tau)}\mathbb{E}\left[\|u_\tau\|_H^2\right]\leq \lim_{t\rightarrow \infty}e^{-\gamma(t-\tau)} R = 0,
	$$
	which implies that there exists $T=T(R)>0$ such that for any $t-\tau \geq T$,
	\begin{align}\label{3DSTNSeq4.11}
		e^{-\gamma(t-\tau)}\mathbb{E}\left[\|u_\tau\|_H^2\right]\leq e^{-\gamma(t-\tau)} R\leq \frac{2}{\gamma}\left(\frac{1}{\nu\lambda}\|g_0\|_{C_b(\mathbb{R},H)}^{2}+\|h_0\|_{C_b(\mathbb{R},H)}^2\right). 
	\end{align}
	Let $\mathcal{M}_1=\frac{4}{\gamma}\left(\frac{1}{\nu\lambda}\|g_0\|_{C_b(\mathbb{R},H)}^{2}+\|h_0\|_{C_b(\mathbb{R},H)}^2\right)+{\bf k}_2$, 
	then the desired result can be obtained. This completes the proof.
\end{proof}

As a direct consequence of Lemma \ref{NSE-dis-lemma4.1*}, we have the following two lemmas.

\begin{lemma}\label{NSE-dis-lemma004.2}
	Under the assumptions of Lemma \ref{NSE-dis-lemma4.1*}, there exists $\mathscr{M}>0$ such that for any $\tau, t \in \mathbb{R}$ with $t\geq \tau$, $\varepsilon\in (0,1]$  and $u_{\tau}\in L^2(\Omega,\mathscr{F}_{\tau};H)$, there is the following estimate
	$$
	\int_{\tau}^{t}\mathbb{E}\left[\|u(s,\tau,u_{\tau})\|_V^2\right]ds \leq \mathscr{M}
	$$
	here $\mathscr{M}$ depends on $\gamma, \nu, |\mathcal{O}|, \|\phi_1\|_{L^\infty(\mathcal{O})},\|\beta\|_{\ell^2},\|\kappa\|_{L^{\infty}(\mathcal{O})}, \|\hat{\gamma}\|_{\ell^2}, g_0, h_0,u_\tau$, but does not depend on $\varepsilon, \tau, u_{\tau}$ and $(g,h)\in \Sigma$.  
\end{lemma}

\begin{lemma}\label{NSE-dis-lemma4.2}
	Under the assumptions of Lemma \ref{NSE-dis-lemma4.1*}, there exists $T=T(R)>1$ such that for any $t-\tau\geq T$ and $\varepsilon\in (0,1]$, the following estimate
	$$
	\int_{t-1}^{t}\mathbb{E}\left[\|u(s,\tau,u_{\tau})\|_V^2\right]ds \leq \mathcal{M}_2
	$$
	holds for every $u_{\tau}\in L^2(\Omega,\mathscr{F}_{\tau};H)$  with $\mathbb{E}\left[\|u_\tau\|_H^2\right]\leq R$, and $\mathcal{M}_1>0$ is a constant that depends on $\gamma, \nu, |\mathcal{O}|, \|\phi_1\|_{L^\infty(\mathcal{O})},\|\beta\|_{\ell^2},\|\kappa\|_{L^{\infty}(\mathcal{O})}, \|\hat{\gamma}\|_{\ell^2}, g_0, h_0$, but does not depend on $\varepsilon, \tau, u_{\tau}$ and $(g,h)\in \Sigma$.  
\end{lemma}
\begin{proof}
	Due to
	\begin{align*}
		\int_{t-1}^{t} \mathbb{E} \left[\| u(s,\tau,u_{\tau})\|_V^2\right]ds\leq e^{\gamma}\int_{t-1}^{t}  e^{-\gamma(t-s)}\mathbb{E} \left[\| u(s,\tau,u_{\tau})\|_V^2\right]ds\\
		\leq e^\gamma\int_{\tau}^{t} e^{-\gamma(\tau-s)}\mathbb{E} \left[\| u(s,\tau,u_{\tau})\|_V^2\right]ds,
	\end{align*}
	which, along with Lemma \ref{NSE-dis-lemma4.1*}, concludes the desired result by introducing $\mathcal{M}_2:= e^\gamma\mathcal{M}_1$.
\end{proof}

Next, we give the uniform boundedness of the solution $u$ of problem \eqref{NSE-dis1.-01} in  $L^4(\Omega,\mathscr{F}_{t}; H)$.
\begin{lemma}\label{NSE-dis-lemma4.3}
	Suppose that $(H_1)-(H_3)$, \eqref{NSE-dis-def4.1} and \eqref{NSE-dis-def4.1*} hold. Then, for any $R>0$, there exists $T=T(R)>0$ such that for all $\tau \in \mathbb{R}$, $t-\tau\geq T$ and $\varepsilon\in (0,1]$, the solution $u$ of problem \eqref{NSE-dis1.-01} satisfies the following inequality:
	\begin{align*}
		\mathbb{E}\left[\|u(t,\tau,u_{\tau})\|_H^4\right]+	\int_{\tau}^t e^{-\gamma(t-s)}\mathbb{E}\left[\|u(s,\tau,u_{\tau})\|_H^2\|u(s,\tau,u_{\tau})\|_V^2\right]ds
		\leq \mathcal{M}_3,
	\end{align*}
	where $u_{\tau}\in L^4(\Omega,\mathscr{F}_{\tau};H)$ with $\mathbb{E}\left[\|u_\tau\|_H^4\right]\leq R$, and $\mathcal{M}_3$ is a positive constant independent of $\varepsilon, \tau, u_{\tau}$ and $(g,h)\in \Sigma$, which may depend on  $\gamma, \nu, |\mathcal{O}|, \|\phi_1\|_{L^\infty(\mathcal{O})},\|\beta\|_{\ell^2},\|\kappa\|_{L^{\infty}(\mathcal{O})}, \|\hat{\gamma}\|_{\ell^2}, g_0, h_0$. 
\end{lemma}
\begin{proof}
	From \eqref{3DSTNSeq4.3} we can obtain that for all $t\geq \tau$,
	\begin{align}\label{3DSTNSeq4.16}
		\begin{split}
			&\,\, e^{2\gamma t}\|u(t)\|_H^4+2\nu\int_\tau^t e^{2\gamma s}\|u(s)\|_H^{2}\|u(s)\|_V^2ds+\left(\nu\lambda-2\gamma\right)\int_\tau^t e^{2\gamma s}\|u(s)\|_H^4ds\\
			&\leq e^{2\gamma\tau}\|u_\tau\|_H^4+\frac{27}{2\gamma \nu^3\lambda^3}\|g_0\|_{C_b(\mathbb{R},H)}^4 e^{2\gamma t} +
			4\int_\tau^t e^{2\gamma s}\|u(s)\|_H^{2} \left(f(\cdot,u(s),\mathscr{L}_{u(s)},u(s))\right)ds \\
			& + 4\varepsilon \int_\tau^t e^{2\gamma s}\|u(s)\|_H^2 \left(u(s),\sigma(s,u(s),\mathscr{L}_{u(s)})\right)dW(s)\\
			&+6\varepsilon^2\int_\tau^t e^{2\gamma s}\|u(s)\|_H^2\|\sigma(s,u(s),\mathscr{L}_{u(s)})\|^2_{\mathcal{L}_2(\ell^2,H)}ds,
		\end{split}
	\end{align}
	$\mathbb{P}$-almost surely, here we used the following inequality
	\begin{align*}
		4\int_\tau^t e^{2\gamma s}\|u(s)\|_H^{2}\left(g(s),u(s)\right)ds
		\leq& \nu\lambda\int_\tau^t  e^{2\gamma s}\|u(s)\|_H^4ds+\frac{27}{\nu^3\lambda^3}\int_\tau^t e^{2\gamma s}\|g(s)\|_H^4ds\\
		\leq & \nu\lambda\int_\tau^t  e^{2\gamma s}\|u(s)\|_H^4ds+\frac{27}{2\gamma \nu^3\lambda^3}\|g_0\|_{C_b(\mathbb{R},H)}^4 e^{2\gamma t}.
	\end{align*}
	Setting $\tau_m=\inf\{t\geq \tau:  \| u(t) \|_H > m\}$. Then, by \eqref{3DSTNSeq4.16} we can derive that for all $t\geq \tau$, 
	\begin{align*}%\label{3DSTNSeq4.17}
		%		\begin{split}
			&\,\, e^{2\gamma (t\wedge \tau_m)}\|u(t\wedge \tau_m)\|_H^4+2\nu\int_\tau^{t\wedge \tau_m} e^{2\gamma s}\|u(s)\|_H^{2}\|u(s)\|_V^2ds+\left(\nu\lambda-2\gamma\right)\int_\tau^{t\wedge \tau_m} e^{2\gamma s}\|u(s)\|_H^4ds\\
			&\leq e^{2\gamma\tau}\|u_\tau\|_H^4+\frac{27}{2\gamma \nu^3\lambda^3}\|g_0\|_{C_b(\mathbb{R},H)}^4 e^{2\gamma ({t\wedge \tau_m})} +
			4\int_\tau^{t\wedge \tau_m} e^{2\gamma s}\|u(s)\|_H^{2} \left(f(\cdot,u(s),\mathscr{L}_{u(s)},u(s))\right)ds \\
			& + 4\varepsilon \int_\tau^{t\wedge \tau_m} e^{2\gamma s}\|u(s)\|_H^2 \left(u(s),\sigma(s,u(s),\mathscr{L}_{u(s)})\right)dW(s)\\
			&+6\varepsilon^2\int_\tau^{t\wedge \tau_m} e^{2\gamma s}\|u(s)\|_H^2\|\sigma(s,u(s),\mathscr{L}_{u(s)})\|^2_{\mathcal{L}_2(\ell^2,H)}ds,
			%		\end{split}
	\end{align*}
	from which we have
	\begin{align}\label{3DSTNSeq4.17}
		\begin{split}
			&\,\, \mathbb{E}\left[e^{2\gamma (t\wedge \tau_m)}\|u(t\wedge \tau_m)\|_H^4\right]+2\nu\mathbb{E}\left[\int_\tau^{t\wedge \tau_m} e^{2\gamma s}\|u(s)\|_H^{2}\|u(s)\|_V^2ds\right]\\
			&+\left(\nu\lambda-2\gamma\right)\mathbb{E}\left[\int_\tau^{t\wedge \tau_m} e^{2\gamma s}\|u(s)\|_H^4ds\right]\\
			&\leq \mathbb{E}\left[e^{2\gamma\tau}\|u_\tau\|_H^4\right]+\frac{27}{2\gamma \nu^3\lambda^3}\|g_0\|_{C_b(\mathbb{R},H)}^4 e^{2\gamma t}\\
			& +
			4\mathbb{E}\left[\int_\tau^{t\wedge \tau_m} e^{2\gamma s}\|u(s)\|_H^{2} \left(f(\cdot,u(s),\mathscr{L}_{u(s)},u(s))\right)ds\right] \\
			&+6\varepsilon^2\mathbb{E}\left[\int_\tau^{t\wedge \tau_m} e^{2\gamma s}\|u(s)\|_H^2\|\sigma(s,u(s),\mathscr{L}_{u(s)})\|^2_{\mathcal{L}_2(\ell^2,H)}ds\right].
		\end{split}
	\end{align}
	For  the third term on the right-hand side of \eqref{3DSTNSeq4.17}, by \eqref{NSE-dis-f_2} we have
	\begin{align}\label{3DSTNSeq4.18}
	%	\begin{split}
			&\,\,4\mathbb{E}\left[\int_\tau^{t\wedge \tau_m} e^{2\gamma s}\|u(s)\|_H^{2} \left(f(\cdot,u(s),\mathscr{L}_{u(s)},u(s))\right)ds\right] \nonumber \\
			%			&\leq 2\mathbb{E}\left[\int_\tau^{t\wedge \tau_m} e^{\gamma s}\int_{\mathcal{O}} \left(|\phi_1(x)||u(s)|+|\phi_1(x)||u(s)|^2+\psi_1(x)\sqrt{\mathbb{E}\left[|u(s)|_2^2\right]}|u(s)|\right)dxds\right] \nonumber\\
			&\leq 4\mathbb{E}\left[\int_\tau^{t\wedge \tau_m} e^{2\gamma s}\|\phi_1\|_{L^2(\mathcal{O})}\|u(s)\|_H^{3} ds\right]+4\mathbb{E}\left[\int_\tau^{t\wedge \tau_m} e^{2\gamma s}\|\phi_1\|_{L^\infty(\mathcal{O})}\|u(s)\|_H^4ds\right]\nonumber\\
			&+ 4\mathbb{E}\left[\int_\tau^{t\wedge \tau_m} e^{2\gamma s}\|\psi_1\|_{L^2(\mathcal{O})}\|u(s)\|_H^{3}\mathbb{E}\left[\|u(s)\|_H^2\right]ds\right]\nonumber\\
			&\leq \frac{\nu \lambda}{4}\mathbb{E}\left[\int_\tau^{t\wedge \tau_m} e^{2\gamma s}\|u(s)\|_H^4ds\right]
			+\frac{27\cdot 64}{ \nu^3\lambda ^3}\mathbb{E}\left[\int_\tau^{t\wedge \tau_m} e^{2\gamma s}\|\phi_1\|_{L^2(\mathcal{O})}^4ds\right]\nonumber\\
			&+4\|\phi_1\|_{L^\infty(\mathcal{O})}\mathbb{E}\left[\int_\tau^{t\wedge \tau_m} e^{2\gamma s}\|u(s)\|_H^4ds\right]+2\|\psi_1\|_{L^2(\mathcal{O})}\mathbb{E}\left[\int_\tau^{t\wedge \tau_m} e^{2\gamma s}\left[\|u(s)\|_H^4\right]ds\right]\nonumber\\ &+2\|\psi_1\|_{L^2(\mathcal{O})}\mathbb{E}\left[\int_\tau^{t\wedge \tau_m} e^{2\gamma s}\|u(s)\|_H^2\mathbb{E}\left[\|u(s)\|_H^2\right]ds\right]\nonumber\\
			&\leq \left(\frac{\nu \lambda}{4}+4\|\phi_1\|_{L^\infty(\mathcal{O})}+2\|\psi_1\|_{L^2(\mathcal{O})}\right)\mathbb{E}\left[\int_\tau^{t\wedge \tau_m} e^{2\gamma s}\|u(s)\|_H^4ds\right]\nonumber\\
			&+2\|\psi_1\|_{L^2(\mathcal{O})}\int_\tau^{t} e^{2\gamma s}\left(\mathbb{E}\left[\|u(s)\|_H^2\right]\right)^2ds+\frac{864}{\gamma \nu^3\lambda ^3}\|\phi_1\|_{L^\infty(\mathcal{O})}^4 |\mathcal{O}|^2e^{2\gamma t}\nonumber\\
			&\leq \left(\frac{\nu \lambda}{4}+4\|\phi_1\|_{L^\infty(\mathcal{O})}+4\|\psi_1\|_{L^2(\mathcal{O})}\right)\int_\tau^{t} e^{2\gamma s}\mathbb{E}\left[\|u(s)\|_H^4\right]ds+\frac{864}{\gamma \nu^3\lambda ^3}\|\phi_1\|_{L^\infty(\mathcal{O})}^4 |\mathcal{O}|^2e^{2\gamma t}.
	%	\end{split}
	\end{align}
	For  the last term on the right-hand side of \eqref{3DSTNSeq4.17}, it follows from \eqref{NSE-dis-G4} that
	\begin{align}\label{3DSTNSeq4.19}
		\begin{split}
			&\,\,6\varepsilon^2\mathbb{E}\left[\int_\tau^{t\wedge \tau_m} e^{2\gamma s}\|u(s)\|_H^2\|\sigma(s,u(s),\mathscr{L}_{u(s)})\|^2_{\mathcal{L}_2(\ell^2,H)}ds\right]\\
			&\leq \frac{144}{\nu \lambda}\mathbb{E}\left[\int_\tau^{t\wedge \tau_m}
			e^{2\gamma s}\|h_0\|_{C_b(\mathbb{R},H)}^4ds\right]
			+24\varepsilon^2 \mathbb{E}\left[\int_\tau^{t\wedge \tau_m}
			e^{2\gamma s}\|\beta\|_{\ell^2}^2\|\kappa\|^4_{L^\infty(\mathcal{O})}|\mathcal{O}|^2ds\right]\\
			&+\left(\frac{\nu \lambda}{4}
			+24\varepsilon^2\|\beta\|_{\ell^2}^2+24\varepsilon^2\|\kappa\|^2_{L^\infty(\mathcal{O})}\|\hat{\gamma}\|_{\ell^2}^2\right)\mathbb{E}\left[\int_\tau^{t\wedge \tau_m}
			e^{2\gamma s}\|u(s)\|_H^4ds\right]\\
			&+48\varepsilon^2\|\beta\|_{\ell^2}^2\|\kappa\|^2_{L^\infty(\mathcal{O})}|\mathcal{O}|\mathbb{E}\left[\int_\tau^{t\wedge \tau_m}
			e^{2\gamma s}\|u(s)\|_H^2\mathbb{E}\left[\|u(s)\|_H^2\right]ds\right]\\
			&\leq \left(\frac{72}{\gamma\nu \lambda}\|h_0\|_{C_b(\mathbb{R},H)}^4
			+\frac{12}{\gamma}\|\beta\|_{\ell^2}^2\|\kappa\|^4_{L^\infty(\mathcal{O})}|\mathcal{O}|^2\right) e^{2\gamma t}\\
			&+\left(\frac{\nu \lambda}{4}
			+24\|\beta\|_{\ell^2}^2+24\|\kappa\|^2_{L^\infty(\mathcal{O})}\|\hat{\gamma}\|_{\ell^2}^2+48\|\beta\|_{\ell^2}^2\|\kappa\|^2_{L^\infty(\mathcal{O})}|\mathcal{O}|\right)\int_\tau^{t}
			e^{2\gamma s}\mathbb{E}\left[\|u(s)\|_H^4\right]ds.
		\end{split}
	\end{align} 
	
	By \eqref{3DSTNSeq4.17}-\eqref{3DSTNSeq4.19} we have
	{\small
	\begin{align*}%\label{3DSTNSeq4.20}
				&\,\, \mathbb{E}\left[e^{2\gamma (t\wedge \tau_m)}\|u(t\wedge \tau_m)\|_H^4\right]+2\nu\mathbb{E}\left[\int_\tau^{t\wedge \tau_m} e^{2\gamma s}\|u(s)\|_H^{2}\|u(s)\|_V^2ds\right]\\
			&+\left(\frac{\nu\lambda}{2}-2\gamma\right)\mathbb{E}\left[\int_\tau^{t\wedge \tau_m} e^{2\gamma s}\|u(s)\|_H^4ds\right]
			\leq e^{2\gamma\tau}\mathbb{E}\left[\|u_\tau\|_H^4\right]\\
			&+\underbrace{\left(\frac{27}{2\gamma \nu^3\lambda^3}\|g_0\|_{C_b(\mathbb{R},H)}^4 +\frac{864}{\gamma \nu^3\lambda ^3}\|\phi_1\|_{L^\infty(\mathcal{O})}^4 |\mathcal{O}|^2+\frac{72}{\gamma\nu \lambda}\|h_0\|_{C_b(\mathbb{R},H)}^4
			+\frac{12}{\gamma}\|\beta\|_{\ell^2}^2\|\kappa\|^4_{L^\infty(\mathcal{O})}|\mathcal{O}|^2\right)}_{{\bf k}_3} e^{2\gamma t}\\
			& +\underbrace{4\left(\|\phi_1\|_{L^\infty(\mathcal{O})}+\|\psi_1\|_{L^2(\mathcal{O})}+6\|\beta\|_{\ell^2}^2+6\|\kappa\|^2_{L^\infty(\mathcal{O})}\|\hat{\gamma}\|_{\ell^2}^2+12\|\beta\|_{\ell^2}^2\|\kappa\|^2_{L^\infty(\mathcal{O})}|\mathcal{O}|\right)}_{{\bf k}_4}\int_\tau^{t}
			e^{2\gamma s}\mathbb{E}\left[\|u(s)\|_H^4\right]ds.
	\end{align*}
}

	Taking the limit as $m \rightarrow \infty$ in the above inequality and applying Fatou's lemma, we find that for all $t \geq \tau$, it holds that
	\begin{align*}
		\begin{split}
			& e^{2\gamma t}\mathbb{E}\left[\|u(t)\|_H^4\right]+2\nu\int_\tau^{t} e^{2\gamma s}\mathbb{E}\left[\|u(s)\|_H^{2}\|u(s)\|_V^2\right]ds+\gamma\int_\tau^{t} e^{2\gamma s}\mathbb{E}\left[\|u(s)\|_H^4\right]ds\\
			\leq& e^{2\gamma\tau}\mathbb{E}\left[\|u_\tau\|_H^4\right]+ {\bf k}_3e^{2\gamma t}+\left(2\gamma-\frac{\nu\lambda}{2}+{\bf k}_4\right)\int_\tau^{t}
			e^{2\gamma s}\mathbb{E}\left[\|u(s)\|_H^4\right]ds\\
			\leq& e^{2\gamma\tau}\mathbb{E}\left[\|u_\tau\|_H^4\right]+ {\bf k}_3e^{2\gamma t}+\left(2\gamma-\frac{\nu\lambda}{2}+{\bf k}_0\right)\int_\tau^{t}
			e^{2\gamma s}\mathbb{E}\left[\|u(s)\|_H^4\right]ds,
		\end{split}
	\end{align*}
	which, along with \eqref{NSE-dis-def4.1}, can get that for all $t\geq \tau$,
	\begin{align}\label{3DSTNSeq4.20}
		\begin{split}
			&\,\, \mathbb{E}\left[\|u(t)\|_H^4\right]+2\nu\int_\tau^{t} e^{2\gamma (s-t)}\mathbb{E}\left[\|u(s)\|_H^{2}\|u(s)\|_V^2\right]ds\leq e^{-2\gamma(t-\tau)}\mathbb{E}\left[\|u_\tau\|_H^4\right]+ {\bf k}_3.
		\end{split}
	\end{align}
	
	Note that $\mathbb{E}\left[\|u_\tau\|_H^4\right]\leq R$, then it holds
	$$
	e^{-2\gamma(t-\tau)}\mathbb{E}\left[\|u_\tau\|_H^4\right]\leq e^{-2\gamma(t-\tau)}R\rightarrow 0, \text{~~as~~} t\rightarrow \infty,
	$$
	and hence there exists $T=T(R)> 0$ such that for all $t-\tau\geq T$, 
	$$
		e^{-2\gamma(t-\tau)}\mathbb{E}\left[\|u_\tau\|_H^4\right]\leq {\bf k}_3,
	$$
	which, together with \eqref{3DSTNSeq4.20}, concludes the desired conclusion. This completes the proof.
\end{proof}

Now, we derive the long-time uniform regularity estimates of solution $u$ of problem \eqref{NSE-dis1.-01} in $L^2(\Omega,\mathscr{F}_{t}; V)$.
\begin{lemma}\label{NSE-dis-lemma4.5}
	Suppose that $(H_1)-(H_3)$, \eqref{NSE-dis-def4.1} and \eqref{NSE-dis-def4.1*} hold. Then, for every $R>0$, there exists $T=T(R)>1$ such that for all $\tau \in \mathbb{R}$, $t-\tau\geq T$ and $\varepsilon\in (0,1]$, the solution $u$ of problem \eqref{NSE-dis1.-01} satisfies
		\begin{align*}
			\mathbb{E}\left[\mathscr{G}(t,\tau,u_{\tau})\| u(t,\tau,u_{\tau})\|_V^2\right]\leq \mathcal{M}_4
		\end{align*}
		with
		\begin{align}\label{NSE-disghjmathfrak{F}}
			\mathscr{G}(t,\tau,u_\tau)=e^{-\frac{27\mathfrak{C}^4}{2\nu^3}\int_\tau^t \|u(s,\tau,u_\tau)\|_H^2 \|u(s,\tau,u_\tau)\|_V^2ds},
		\end{align}
		where $u_{\tau}\in L^2(\Omega,\mathscr{F}_{\tau};H)$  with $\mathbb{E}\left[\|u_\tau\|_H^2\right]\leq R$, $\gamma>0$ is the same number as in \eqref{NSE-dis-def4.1*}. Particularly, here $\mathcal{M}_4>0 $ is a constant that depends on $\gamma, \nu,{\bf \hat{c}}, |\mathcal{O}|$, $\|\phi_1\|_{L^\infty(\mathcal{O})}$, $\|\beta\|_{\ell^2}$, $\|\kappa\|_{L^{\infty}(\mathcal{O})}$, $\|\nabla\kappa\|_{L^{\infty}(\mathcal{O})}$, $\|\hat{\gamma}\|_{\ell^2}, \|L\|_{\ell^2}, g_0, h_0$, but does not depend on $\varepsilon, \tau, u_{\tau}$ and $(g,h)\in \Sigma$. 
	\end{lemma}
\begin{proof} 
	We shall derive the long-time uniform estimate of the solution in a formal manner, which can be rigorously justified via a limiting argument.
	% for the process $\|u(t,\tau,u_{\tau})\|_{V}^2$
	By \eqref{NSE-dis1.-01} and It\^{o}'s formula, we can get that for all $\tau\in \mathbb{R}$, $t-\tau>1$ and $\varrho \in (t-1,t)$,
	\begin{align}\label{NSE-disghj4.15}
		\begin{split}
			&\,\, \| u(t,\tau,u_{\tau})\|_V^2
			+2\nu\int_\varrho^t \|A u(s,\tau,u_{\tau})\|_{H}^2 ds=\| u(\varrho,\tau,u_{\tau})\|_V^2\\
			&+2\int_\varrho^t \left(g(s),Au(s,\tau,u_{\tau})\right)ds+2\int_\varrho^t  \left(f(\cdot,u(s,\tau,u_{\tau}),\mathscr{L}_{u(s,\tau,u_{\tau})}),Au(s,\tau,u_{\tau})\right)ds\\
			&-2\int_\varrho^t  \langle B(u(s,\tau,u_{\tau}),u(s,\tau,u_{\tau})), Au(s,\tau,u_{\tau})\rangle ds\\
			&+2\varepsilon\int_\varrho^t  \left(Au(s,\tau,u_{\tau}),\sigma(s,u(s,\tau,u_{\tau}),\mathscr{L}_{u(s,\tau,u_{\tau})})\right)d W(s)\\
			&+\varepsilon^2 \int_\varrho^t \|\nabla \sigma(s,u(s,\tau,u_{\tau}),\mathscr{L}_{u(s,\tau,u_{\tau})})\|_{\mathcal{L}_2(\ell^2,H)}^2ds.
		\end{split}
	\end{align}
	
	We now handle the right-hand side terms of \eqref{NSE-disghj4.15}. For  the second term on the right-hand side of \eqref{NSE-disghj4.15}, it is easy to obtain that
		\begin{align}\label{NSE-disghj4.16}
		\begin{split}
			2\int_\varrho^t \left(g(s),Au(s,\tau,u_{\tau})\right)ds\leq \frac{\nu}{4} \int_\varrho^t  \|A u(s,\tau,u_{\tau})\|^2_{H} ds + \frac{4}{\nu}\int_\varrho^t  \|g(s)\|_{H}^2ds.
		\end{split}
	\end{align}
	For the third term on the right-hand side of \eqref{NSE-disghj4.15}, by H\"{o}lder's inequality, Young's inequality and \eqref{NSE-dis-f_2} we have
	\begin{align}\label{NSE-disghj4.17}
		&2\int_\varrho^t  \left(f(\cdot,u(s,\tau,u_{\tau}),\mathscr{L}_{u(s,\tau,u_{\tau})}),Au(s,\tau,u_{\tau})\right)ds\nonumber\\
		\leq&  \frac{\nu}{4}\int_\varrho^t  \|Au(s,\tau,u_{\tau})\|_{H}^2ds+\frac{4}{\nu}\int_\varrho^t  \|f(x,u(s,\tau,u_{\tau}),\mathscr{L}_{u(s,\tau,u_{\tau})})\|_{H}^2ds\nonumber\\
		\leq& \frac{\nu}{4}\int_\varrho^t  \|Au(s,\tau,u_{\tau})\|_{H}^2ds+\frac{8}{\nu}\int_\varrho^t \int_{\mathcal{O}} \left|\phi_1(x)\right|^2(1+|u(s,\tau,u_{\tau})|)^2dxds\nonumber\\
		&+\frac{8}{\nu}\int_\varrho^t \int_{\mathcal{O}}\left|\psi_1 (x)\right|^2 {\mathbb{E}(\|u(s,\tau,u_{\tau})\|_{H}^2)}dxds\nonumber\\
		\leq& \frac{\nu}{4}\int_\varrho^t  \|Au(s,\tau,u_{\tau})\|_{H}^2ds
		+\frac{16}{\nu\lambda}\left\|\phi_1\right\|_{L^\infty(\mathcal{O})}\int_\varrho^t \|u(s,\tau,u_{\tau})\|_V^2ds\nonumber\\
		& +\frac{8}{\nu}\left\|\psi_1\right\|^2_{L^2(\mathcal{O})}\int_\varrho^t \mathbb{E}(\|u(s,\tau,u_{\tau})\|_H^2)ds+\frac{16}{\nu}|\mathcal{O}|\|\phi_1\|_{L^\infty(\mathcal{O})}^2.
	\end{align}
	For the fourth term on the right-hand side of \eqref{NSE-disghj4.15}, by \eqref{NSE-dis3.2} and Young's inequality we get
	\begin{align}\label{NSE-disghj4.18}
		\begin{split}
			&-2\int_\varrho^t  \langle B(u(s,\tau,u_{\tau}),u(s,\tau,u_{\tau})), Au(s,\tau,u_{\tau})\rangle ds\\
%\leq& 2\int_\varrho^t  \left|\langle B(u(s,\tau,u_{\tau}),u(s,\tau,u_{\tau})), Au(s,\tau,u_{\tau})\rangle\right| ds\\
%			\leq& 2\mathfrak{C}\int_\varrho^t \|A u(s,\tau,u_{\tau})\|_{H}^{3/2} \|u(s,\tau,u_{\tau})\|_H^{1/2}\| u(s,\tau,u_{\tau})\|_V ds\\
			\leq& \frac{\nu}{2}\int_\varrho^t  \|Au(s,\tau,u_{\tau})\|_{H}^2ds+\frac{27\mathfrak{C}^4}{2\nu^3}\int_\varrho^t  \|u(s,\tau,u_{\tau})\|_H^2 \| u(s,\tau,u_{\tau})\|_V^{4}ds.
		\end{split}
	\end{align}
	For the sixth term on the right-hand side of \eqref{NSE-disghj4.15}. By \eqref{NSE-dis-gh_03}, \eqref{NSE-dis-G1} and \eqref{NSE-dis-G3} we have
	\begin{align}\label{NSE-disghj4.19}
%		\begin{split}
			&\varepsilon^2 \int_\varrho^t \|\nabla \sigma(s,u(s,\tau,u_{\tau}),\mathscr{L}_{u(s,\tau,u_{\tau})})\|_{\mathcal{L}_2(\ell^2,H)}^2ds\nonumber \\
			\leq&  2\varepsilon^2 \int_\varrho^t  \|\nabla h(s)\|_H^2 ds+
			2\varepsilon^2 \sum_{k=1}^\infty \int_\varrho^t  \int_{\mathcal{O}}  \left|\nabla\kappa(x)\right|^2 \left|\sigma_k(s,u(s,\tau,u_{\tau}),\mathscr{L}_{u(s,\tau,u_{\tau})})\right|^2 dxds\nonumber \\
			&+2\varepsilon^2 \sum_{k=1}^\infty \int_\varrho^t  \int_{\mathcal{O}}  \left|\kappa(x)\right|^2 \left|\nabla u(s,\tau,u_{\tau})\right|^2 \left|\frac{\partial\sigma_k}{\partial u}(s,u(s,\tau,u_{\tau}),\mathscr{L}_{u(s,\tau,u_{\tau})})\right|^2 dxds\nonumber \\
			\leq& 2\varepsilon^2 \int_{t-1}^t  \|\nabla h(s)\|_H^2 ds+2\varepsilon^2 \left(\sum_{k=1}^\infty L_k^2\right) \int_\varrho^t \int_{\mathcal{O}} \left|\kappa(x)\right|^2\left|\nabla u(s,\tau,u_{\tau})\right|^2dxds\nonumber \\
			&+
			4\varepsilon^2 \sum_{k=1}^\infty \int_\varrho^t  \int_{\mathcal{O}}  \left|\nabla\kappa(x)\right|^2 \left(\beta_{k}^2\left(1+\sqrt{\mathbb{E}\left[\|u(s,\tau,u_{\tau})\|_H^2\right]}\right)^2+\hat{\gamma}_k^2|u(s,\tau,u_{\tau})|^2\right) dxds\nonumber \\
			\leq& 2{\bf \hat{c}}+2 \|L\|_{\ell^2}^2 \|\kappa\|_{L^\infty(\mathcal{O})}^2 \int_\varrho^t \|\nabla u(s,\tau,u_{\tau})\|_H^2ds+8\|\beta\|_{\ell^2}^2\|\nabla\kappa\|_{L^2(\mathcal{O})}^2\int_\varrho^t  ds\nonumber \\
			&+8\|\beta\|_{\ell^2}^2\|\nabla\kappa\|_{L^2(\mathcal{O})}^2 \int_\varrho^t  \mathbb{E}\left[\|u(s,\tau,u_{\tau})\|_H^2\right]ds
			+4\|\hat{\gamma}\|_{\ell^2}^2\|\nabla\kappa\|_{L^\infty(\mathcal{O})}^2  \int_\varrho^t  \|u(s,\tau,u_{\tau})\|_H^2ds\nonumber \\
			\leq & 2{\bf \hat{c}}+8|\mathcal{O}|\|\beta\|_{\ell^2}^2\|\nabla\kappa\|_{L^\infty(\mathcal{O})}^2+2\left(\|L\|_{\ell^2}^2 \|\kappa\|_{L^\infty(\mathcal{O})}^2+\frac{2}{\lambda}\|\hat{\gamma}\|_{\ell^2}^2\|\nabla\kappa\|_{L^\infty(\mathcal{O})}^2\right)\int_\varrho^t \|u(s,\tau,u_{\tau})\|_V^2ds\nonumber \\
			&+8|\mathcal{O}|\|\beta\|_{\ell^2}^2\|\nabla\kappa\|_{L^\infty(\mathcal{O})}^2 \int_\varrho^t  \mathbb{E}\left[\|u(s,\tau,u_{\tau})\|_H^2\right]ds.
%			\leq&  \|L\|_{\ell^2}^2\int_\varrho^t  \left\|u(s,\tau,u_{\tau})\right\|_V^2ds.
%		\end{split}
	\end{align}

	Combining \eqref{NSE-disghjmathfrak{F}} and \eqref{NSE-disghj4.15}, it gets
	\begin{align*}
%		\begin{split}
			&~~\mathscr{G}(t,\tau,u_{\tau})\| u(t,\tau,u_{\tau})\|_V^2
			+2\nu\int_\varrho^t  \mathscr{G}(s,\tau,u_{\tau}) \|Au(s,\tau,u_{\tau})\|_H^2 ds\\
			&=\mathscr{G}(\varrho,\tau,u_{\tau})\| u(\rho,\tau,u_{\tau})\|_V^2-\frac{27\mathfrak{C}^4}{2\nu^3}\int_\varrho^t \mathscr{G}(s,\tau,u_{\tau})\|u(s,\tau,u_{\tau})\|_H^2\| u(s,\tau,u_{\tau})\|_V^{4}ds\\
			&+2\int_\varrho^t \mathscr{G}(s,\tau,u_{\tau})\left(g(s),Au(s,\tau,u_{\tau})\right)ds\\
			&-2\int_\varrho^t  \mathscr{G}(s,\tau,u_{\tau})\langle B(u(s,\tau,u_{\tau}),u(s,\tau,u_{\tau})), Au(s,\tau,u_{\tau})\rangle ds\\
			&+2\int_\varrho^t  \mathscr{G}(s,\tau,u_{\tau})\left(f(\cdot,u(s,\tau,u_{\tau}),\mathscr{L}_{u(s,\tau,u_{\tau})}),Au(s,\tau,u_{\tau})\right)ds\\
			&+2\varepsilon\int_\varrho^t  \mathscr{G}(s,\tau,u_{\tau})\left(Au(s,\tau,u_{\tau}),\sigma(s,u(s,\tau,u_{\tau}),\mathscr{L}_{u(s,\tau,u_{\tau})})\right)dW(s)\\
			&+\varepsilon^2 \int_\varrho^t \mathscr{G}(s,\tau,u_{\tau})\|\nabla \sigma(s,u(s,\tau,u_{\tau}),\mathscr{L}_{u(s,\tau,u_{\tau})})\|_{\mathcal{L}_2(\ell^2,H)}^2ds,
%		\end{split}
	\end{align*}
	which, together with
	\eqref{NSE-disghj4.16}-\eqref{NSE-disghj4.19}, deduces that 
	for all $\tau\in \mathbb{R}$, $t-\tau> 1$ and $\varrho\in (t-1,t)$,
	\begin{align}\label{NSE-disghj4.20}
	%	\begin{split}
			&\,\,\, \mathbb{E}\left[\mathscr{G}(t,\tau,u_{\tau})\| u(t,\tau,u_{\tau})\|_V^2\right]+\nu\int_\varrho^t  \mathbb{E}\left[\mathscr{G}(s,\tau,u_{\tau}) \|Au(s,\tau,u_{\tau})\|_H^2\right] ds\nonumber \\
			&\leq \mathbb{E}\left[\mathscr{G}(\varrho,\tau,u_{\tau})\| u(\rho,\tau,u_{\tau})\|_V^2\right]+\frac{4}{\nu}\int_\varrho^t  \mathbb{E}\left[\mathscr{G}(s,\tau,u_{\tau})\|g(s)\|_{H}^2\right]ds\nonumber \\
			&+\underbrace{2\left(\frac{8}{\nu\lambda}\left\|\phi_1\right\|_{L^\infty(\mathcal{O})}+\|L\|_{\ell^2}^2 \|\kappa\|_{L^\infty(\mathcal{O})}^2+\frac{2}{\lambda}\|\hat{\gamma}\|_{\ell^2}^2\|\nabla\kappa\|_{L^\infty(\mathcal{O})}^2\right)}_{{\bf k}_5}\int_\varrho^t \mathbb{E}\left[\mathscr{G}(s,\tau,u_{\tau})\|u(s,\tau,u_{\tau})\|_V^2\right]ds\nonumber \\
			&+\underbrace{8\left(\frac{1}{\nu}\left\|\psi_1\right\|^2_{L^2(\mathcal{O})}+|\mathcal{O}|\|\beta\|_{\ell^2}^2\|\nabla\kappa\|_{L^\infty(\mathcal{O})}^2 \right)}_{{\bf k}_6}
			\int_\varrho^t 
			\mathbb{E}\left[\mathscr{G}(s,\tau,u_{\tau})\right]\mathbb{E}(\|u(s,\tau,u_{\tau})\|_H^2)ds\nonumber \\
			&+\underbrace{\left(\frac{16}{\nu}|\mathcal{O}|\|\phi_1\|_{L^\infty(\mathcal{O})}^2+2{\bf \hat{c}}+8|\mathcal{O}|\|\beta\|_{\ell^2}^2\|\nabla\kappa\|_{L^\infty(\mathcal{O})}^2\right)}_{{\bf k}_7}\int_\varrho^t  \mathbb{E}\left[\mathscr{G}(s,\tau,u_{\tau})\right]ds.
	%	\end{split}
	\end{align}
	
	We integrate \eqref{NSE-disghj4.20} with respect to $\varrho$ from $t-1$ to $t$, it yields
	\begin{align}\label{NSE-disghj4.21}
		\begin{split}
			&\,\,\, \mathbb{E}\left[\mathscr{G}(t,\tau,u_{\tau})\| u(t,\tau,u_{\tau})\|_V^2\right]\\
			&\leq \left(1+{{\bf k}_5}+\frac{{{\bf k}_6}}{\lambda}\right)\int_{t-1}^{t}\mathbb{E}\left[\| u(s,\tau,u_{\tau})\|_V^2\right]ds+{{\bf k}_7}+\frac{4}{\nu}\|g_0\|^2_{C_b{(\mathbb{R},H)}}.
		\end{split}
	\end{align}
	Therefore, the desired result follows by combining Lemma \ref{NSE-dis-lemma4.2} with inequality \eqref{NSE-disghj4.21}. This completes the proof.
\end{proof}

\section{Existence of uniform measure attractors}\label{ex-pull-NS}

In this section, we focus on studying the existence and uniqueness of uniform measure attractors of problem \eqref{NSE-dis1.-01} in $(\mathcal {P}_4 (X), d_{\mathcal{P}(X)})$. To this end, we first define a process on $(\mathcal {P}_4 (X), d_{\mathcal{P}(X)})$. Specifically, for given $t\geq \tau \in\mathbb{R}$, we define $P_*^{(g,h)}:\mathcal{P}_4(H)\rightarrow \mathcal{P}_4(H)$ by
\begin{align}\label{NSE-disghj5.1}
	P_*^{(g,h)}(t,\tau)\mu=\mathscr{L}_{u^{(g,h)}(t,\tau,u_\tau)}, \text{~~for every~} \mu\in \mathcal{P}_4(H),
\end{align}
%where $u^{(g,h)}(t,\tau, u_\tau)$ is the solution of problem \eqref{NSE-dis1.-01} with
%$u_\tau\in L^4(\Omega,\mathscr{F}_\tau, H)$ such that  ${\mathscr{L}}_{u_\tau}
%=\mu$. 
where $u^{(g,h)}(t,\tau, u_\tau)$ is the solution of problem \eqref{NSE-dis1.-01}. For a given \(\mu \in P_4(H)\), one may use Skorokhod's theorem (see \cite{Schmalfuss-1991}) to find another probability space \((\widetilde{\Omega},\widetilde{\mathscr{F}},\widetilde{\mathbb{P}})\) and \(\widetilde{u}_\tau \in L^4(\widetilde{\Omega}, \widetilde{\mathscr{F}}_\tau, H)\) such that \(\mathcal{L}_{\widetilde{u}_\tau} = \mu\). Without loss of generality, we work on the original probability space and denote the random initial values by \(u_\tau \in L^4(\Omega, \mathscr{F}_\tau, H)\) with \(\mathcal{L}_{u_\tau} = \mu\).

In addition, for every $t\in \mathbb{R}^+$ and $\tau \in\mathbb{R}$, we define $U^{(g,h)}(t,\tau):\mathcal{P}_4(H)\rightarrow \mathcal{P}_4(H)$ as follows:
\begin{align}\label{NSE-disghj5.2}
	U^{(g,h)}(t+\tau,\tau) \mu=P_*^{(g,h)}(t+\tau,\tau)\mu,\quad \forall \mu\in \mathcal{P}_4(H).
\end{align}
From the uniqueness of the solutions of problem \eqref{NSE-dis1.-01}, we know that for all $\tau \in \mathbb{R}$, $\tau\leq s \leq t$ and $u_\tau \in L^4(\Omega, \mathscr{F}_\tau; H)$,
$$
u^{(g,h)}(t, \tau, u_\tau)
= u^{(g,h)}(t, s, u^{(g,h)}(s, \tau, u_\tau) ),
$$
it follows that for all $t\geq s\geq \tau \in\mathbb{R}$,
$$
U^{(g,h)} (t, \tau)\mu =U^{(g,h)} (t, s)\circ U^{(g,h)}(s, \tau) \mu, \quad \forall \mu\in \mathcal{P}_4(H).
$$
and for all $\tau \in \mathbb{R}$,
$$
U^{(g,h)}(\tau, \tau) = I_{\mathcal{P}_4(H)}.
$$
Moreover, the similar argument as that of \cite[Lemma 4.1]{li2021b-JDE} we have the following translation
identity for the operator family $\{U^{(g,h)}(t,\tau)\}_{(g,h)\in \Sigma}$ and the translation group $\{T(s)\}_{s\in \mathbb{R}}$:
\begin{align*}
	U^{(g,h)} (t+s, \tau+s)=U^{T(s)(g,h)} (t, \tau),\quad \forall t\geq\tau, \tau\in\mathbb{R}.
\end{align*}

%we need to prove that $\Psi$ forms a  nonautonomous dynamical system in $(\mathcal {P}_4 (X), d_{\mathcal{P}(X)})$. 
In what follows, we prove the weak continuity of
$U^{(g,h)} (t, \tau)$ over bounded subsets of $\mathcal{P}_4(H)\times\Sigma$, which shall be used to establish the joint continuity of the family of processes $\{U^{(g,h)}(t,\tau)\}_{(g,h)\in \Sigma}$.
\begin{lemma}\label{NSE-disghjsss5.1}
	Suppose that $(H_1)-(H_3)$ and \eqref{NSE-dis-def4.1} and \eqref{NSE-dis-def4.1*} hold.  Let $u_\tau, u_{\tau}^n  \in
	L^4(\Omega, {\mathscr{F}}_\tau,H)$ such that 
	\begin{align*}
	\mathbb{E}\left[\|u_\tau\|_H^4 \right]	\leq R \text{~~and~~} 	\mathbb{E}\left[\|u_{\tau}^n\|_H^4 \right)\leq R
	\end{align*}
	for  some $R>0$. If $\mathscr{L}_{u_{\tau}^n}
	\to \mathscr{L}_{u_\tau}$  weakly and $(g_n,h_n)\rightarrow (g,h)$ in $\Sigma$, then for every $\varepsilon \in (0,1]$, $ \tau \in \mathbb{R} $
	and $t\geq \tau$, it holds that $\mathscr{L}_{u^{(g_n,h_n)}(t,\tau, u_{\tau}^n)} \to \mathscr{L}_{u^{(g,h)}(t,\tau, u_\tau) }$ weakly.
\end{lemma}
\begin{proof}
	Since $\mathscr{L}_{u_{\tau}^n}\rightarrow \mathscr{L}_{u_\tau}$ weakly, it follows from Skorokhod's theorem that there exist a probability space $(\widetilde{\Omega}, \widetilde{\mathscr{F}}, \widetilde{\mathbb{P}} )$ and the random variables
	$ \widetilde{u}_\tau $ and $ \widetilde{u}_{\tau}^n $ defined on $(\widetilde{\Omega}, \widetilde{\mathscr{F}}, \widetilde{\mathbb{P}} )$ such that the distributions of $ \widetilde{u}_\tau $ and $ \widetilde{u}_{\tau}^n $ coincide with those of $ {u}_\tau $ and $ {u}_{\tau}^n $, respectively. Moreover,  $ \widetilde{u}_{\tau}^n \to  \widetilde{u}_\tau $ holds, $\widetilde{\mathbb{P}}$-almost surely. In particular,  we note that $ \widetilde{u}_\tau $, $ \widetilde{u}_{\tau}^n $ and $W$ 
	can be viewed as random variables defined on the product space $(\Omega \times \widetilde{\Omega}, \mathscr{F} \times \widetilde{\mathscr{F}},
	{\mathbb{P}} \times \widetilde{\mathbb{P}})$. Hence, we may consider the solutions of the corresponding stochastic equation on $(\Omega \times \widetilde{\Omega}, \mathscr{F} \times \widetilde{\mathscr{F}},
	{\mathbb{P}} \times \widetilde{\mathbb{P}})$ with initial data $ \widetilde{u}_\tau $ and $ \widetilde{u}_{\tau}^n $, rather than the solutions on the original space $(\Omega  , \mathscr{F}  ,  {\mathbb{P}}  )$  with initial data
	$   {u}_\tau$ and $ {u}_\tau^n $. 
	However, for the sake of simplicity, we shall identify the new random variables with the original ones and work exclusively with the solutions of the equation in the original space. Due to $ \widetilde{u}_{\tau}^n \to  \widetilde{u}_\tau $, ${\mathbb{P}}$-almost surely, then, without loss of generality, we may assume that the original sequence satisfies $ {u}_\tau^n \to {u}_\tau$,  $\mathbb{P}$-almost surely.
	
	Assume that
	$ u^n (t) = u^{(g_n,h_n)}( t,\tau, u_{\tau}^n ) $, 
	$ u(t) = u^{(g,h)}( t,\tau, u_{\tau} )$. Let $ \varpi^n (t) = u^{(g_n,h_n)}  (t,\tau,  u_{\tau}^n )
	- u^{(g,h)}  ( t,\tau, u_{\tau} ) $, then for all $t> \tau$, $\varpi^n (t)$ satisfies 
	\begin{align*}
		&d\varpi^n (t)+\nu A\varpi^n (t)dt + \left(B(u^n (t),u^n (t))-B(u(t),u(t))\right)dt=(g_n(t)-g(t))dt\\
		+&\left(f(x,u^n (t),\mathscr{L}_{u^n(t)})-f(x,u (t),\mathscr{L}_{u(t)})\right)dt
		+\varepsilon \left(\sigma(t,u^n (t),\mathscr{L}_{u^n(t)})-\sigma(t,u (t),\mathscr{L}_{u(t)})\right)dW(t)
	\end{align*} 
	with initial data $\varpi^n_\tau:=\varpi^n(\tau)={u}_\tau^n-{u}_\tau$.

	Applying It\^{o}'s formula for the process $\|\varpi^n (t)\|_H^2$, we can obtain that for any  $t> \tau$, $s\in (\tau, t)$ and almost all $\omega$,
	\begin{align}\label{NSE-disghjmea5.3*}
		\begin{split}
			& \|\varpi^n (s)\|_H^2
			+2\nu\int_\tau^s \|\varpi^n (r)\|_V^2dr 
			=\|\varpi^n_\tau\|_H^2\\
			&+2\int_{\tau}^s\left(g_n(r)-g(r),\varpi^n(r)\right)dr-2\int_{\tau}^s \langle B(u^n (r),u^n (r))-B(u(r),u(r)), \varpi^n(r)\rangle dr \\
			& +2\int_{\tau}^s \left(f(\cdot,u^n (r),\mathscr{L}_{u^n(r)})-f(\cdot,u (r),\mathscr{L}_{u(r)}),\varpi^n(r)\right)dr\\
			& +2\varepsilon \int_{\tau}^s \left(\varpi^n(r),(\sigma(r,u^n (r),\mathscr{L}_{u^n(r)})-\sigma(r,u (r),\mathscr{L}_{u(r)}))dW(r)\right)\\ 
			& +\varepsilon^2 \int_{\tau}^s\|\sigma(r,u^n (r),\mathscr{L}_{u^n(r)})-\sigma(r,u (r),\mathscr{L}_{u(r)})\|_{\mathcal{L}_2(\ell^2,H)}^2dr.
		\end{split}
	\end{align}
	
   We deal with the right-hand side terms of \eqref{NSE-disghjmea5.3*}. For the second term on the right-hand side of \eqref{NSE-disghjmea5.3*}, by H\"{o}lder's inequality, Young's inequality  and Poincar\'{e}'s inequality we get
   \begin{align}
   	2\int_{\tau}^s\left(g_n(r)-g(r),\varpi^n(r)\right)dr
   	\leq \frac{\nu }{2}\int_{\tau}^s \|\varpi^n(r)\|_V^2dr+\frac{2}{\nu \lambda}\int_{\tau}^s\|g_n(r)-g(r)\|_H^2dr
   \end{align}
   For the third term on the right-hand side of \eqref{NSE-disghjmea5.3*}, by \eqref{NSE-dis3.--002} and  \eqref{NSE-dis3.2} we have
	\begin{align}\label{NSE-disghjmea5.4}
		\begin{split}
			&-2\int_{\tau}^s \langle B(u^n (r),u^n (r))-B(u(r),u(r)), \varpi^n(r)\rangle dr\\
			%&=-2\int_{\tau}^s \langle B(\varpi^n(r),u(r)), \varpi^n(r) \rangle dr\\
			%			\leq & 2\mathfrak{C}\int_{\tau}^t {}\|\nabla u(s)\|_{H} \|\varpi^n(s)\|_{H} \|\nabla \varpi^n(s)\|_{H}ds\\
			&\leq  \frac{\nu}{2}\int_{\tau}^s \|\varpi^n(r)\|_V^2dr
			+\frac{2\mathfrak{C}^2}{\nu}\int_{\tau}^s\|u(r)\|_V^2 \|\varpi^n(r)\|_H^2dr.
		\end{split}
	\end{align}
	For the fourth term on the right-hand side of \eqref{NSE-disghjmea5.3*}, it follows from \eqref{NSE-dis-f_3} that
	\begin{align}\label{NSE-disghjmea5.5}
		\begin{split}
			&2\int_{\tau}^s {}\left(f(\cdot,u^n (r),\mathscr{L}_{u^n(r)})-f(\cdot,u (r),\mathscr{L}_{u(r)}),\varpi^n(r)\right)dr\\
			\leq&  2\int_{\tau}^s {}\int_{\mathcal{O}} \left(\phi_2(x)|\varpi^n(r)|+\psi_2(x)\sqrt{\mathbb{E}\left[\|\varpi^n(r)\|_H^2\right]}|\varpi^n(r)|\right)dx dr\\
			\leq& \left(2\|\phi_2\|_{L^\infty(\mathcal{O})}+\|\psi_2\|_{L^2(\mathcal{O})}\right)\int_{\tau}^s {} \|\varpi^n(r)\|_H^2dr+ \|\psi_2\|_{L^2(\mathcal{O})}\int_{\tau}^s \mathbb{E}\left[\|\varpi^n(r)\|_H^2\right]dr.
		\end{split}
	\end{align}
	For the fifth term on the right-hand side of \eqref{NSE-disghjmea5.3*}, by \eqref{NSE-dis-G5} we get
	\begin{align}\label{NSE-disghjmea5.6}
		\begin{split}
			& \varepsilon^2 \int_{\tau}^s\|\sigma(r,u^n (r),\mathscr{L}_{u^n(r)})-\sigma(r,u (r),\mathscr{L}_{u(r)})\|_{\mathcal{L}_2(\ell^2,H)}^2dr
			\leq  2\int_{\tau}^s\|h_n(r)-h(r)\|_H^2dr\\
			&\qquad+4\|\kappa\|^2_{L^\infty(\mathcal{O})}\|L\|_{\ell^2}^2(1+|\mathcal{O}|)\int_{\tau}^s \left(\|\varpi^n(r)\|_H^2+\mathbb{E}\left[\|\varpi^n(r)\|_H^2\right]\right)dr.%\\
%			\leq&  2\left(1+|\mathcal{O}|\right)\left(1+4\|\kappa\|^2_{L^\infty(\mathcal{O})}\|L\|_{\ell^2}^2\right)\int_{\tau}^s {}\|\varpi^n(r)\|_H^2dr
%			+ 4\|\kappa\|^2_{L^\infty(\mathcal{O})}\|L\|_{\ell^2}^2(1+|\mathcal{O}|)\int_{\tau}^s\mathbb{E}\left[\|\varpi^n(r)\|_H^2\right]dr.
		\end{split}
	\end{align}
	
	Given $k>0$, define the sequence of stopping times by
	$$
	\tau_k=\inf\left\{s\geq \tau:\int_\tau^s \|u(r)\|_V^2dr>k \right\}
	$$
	with the convention that $\tau_k=+\infty$ if the set is empty.
	By virtue of  \eqref{NSE-disghjmea5.3*}-\eqref{NSE-disghjmea5.6}, we can derive that for any $\tau \in \mathbb{R}$, $t> \tau$, $\tau<s<s_1<t$ and almost all $\omega$,
	\begin{align*}
		&\sup_{\varkappa \in [\tau,s\wedge \tau_k]} \|\varpi^n (\varkappa)\|_H^2+{\nu}\int_\tau^{s \wedge \tau_k} 
		\|\varpi^n (r)\|_V^2dr\\
		\leq& \|\varpi^n_\tau\|_H^2+\underbrace{\left(
			2\|\phi_2\|_{L^\infty(\mathcal{O})}+\|\psi_2\|_{L^2(\mathcal{O})}+4\|\kappa\|^2_{L^\infty(\mathcal{O})}\|L\|_{\ell^2}^2(1+|\mathcal{O}|)\right)}_{{\bf k}_8}\int_{\tau}^{s \wedge \tau_k} \|\varpi^n(r)\|_H^2dr\\
		&+\frac{2\mathfrak{C}^2}{\nu}\int_{\tau}^{s \wedge \tau_k} \|u(r)\|_V^2 \|\varpi^n(r)\|_H^2dr+\frac{2}{\nu \lambda}\int_{\tau}^{s \wedge \tau_k}\|g_n(r)-g(r)\|_H^2dr+2\int_{\tau}^{s \wedge \tau_k}\|h_n(r)-h(r)\|_H^2dr
		\\
		&+ \underbrace{\left(\|\psi_2\|_{L^2(\mathcal{O})}+ 4\|\kappa\|^2_{L^\infty(\mathcal{O})}\|L\|_{\ell^2}^2(1+|\mathcal{O}|)\right)}_{{\bf k}_9}\int_{\tau}^{s \wedge \tau_k} \mathbb{E}\left[\|\varpi^n(r)\|_H^2\right]dr \\
		&+2 \varepsilon\sup_{\varkappa\in [\tau,s\wedge \tau_k]}\left|\int_{\tau}^\varkappa \left(\varpi^n(r), (\sigma(r,u^n (r),\mathscr{L}_{u^n(r)})-\sigma(r,u (r),\mathscr{L}_{u(r)}))dW(r)\right)\right|\\
		\leq& \|\varpi^n_\tau\|_H^2+{\bf k}_8\int_{\tau}^{s } \sup_{\varkappa\in [\tau,r\wedge \tau_k]}\|\varpi^n(\varkappa)\|_H^2dr  + {\bf k}_9\int_{\tau}^{s_1\wedge \tau_k} \mathbb{E}\left[\|\varpi^n(r)\|_H^2\right]dr\\
		&+\frac{2\mathfrak{C}^2}{\nu}\int_{\tau}^{s} \|u(r\wedge \tau_k)\|_V^2 \left(\sup_{\varkappa\in [\tau,r\wedge \tau_k]} \|\varpi^n(\varkappa)\|_H^2\right)dr+{2}\|g_n-g\|_{C_b(\mathbb{R},H)}^2+2\|h_n-h\|_{C_b(\mathbb{R},H)}^2\\
		& +2 \sup_{\varkappa\in [\tau,s_1\wedge \tau_k]}\left|\int_{\tau}^\varkappa \left(\varpi^n(r), (\sigma(r,u^n (r),\mathscr{L}_{u^n(r)})-\sigma(r,u (r),\mathscr{L}_{u(r)}))dW(r)\right)\right|.
	\end{align*}
Applying Gronwall’s lemma to the inequality above yields, for any $\tau \in \mathbb{R}$, $t> \tau$, $\tau<s<s_1<t$ and almost all $\omega$,
	\begin{align}\label{NSE-disghjmea5.07}
		\begin{split}
			&\,\,\sup_{\varkappa\in [\tau,s\wedge \tau_k]} \|\varpi^n (\varkappa)\|_H^2\\
			&\leq \bigg(\|\varpi^n_\tau\|_H^2+\mathbf{k}_{9}\int_{\tau}^{{s_1} \wedge \tau_k} \mathbb{E}\left[\|\varpi^n(r)\|_H^2\right]dr+{2}\|g_n-g\|_{C_b(\mathbb{R},H)}^2+2\|h_n-h\|_{C_b(\mathbb{R},H)}^2\\
			&\quad +2 \sup_{\varkappa\in [\tau,{s_1}\wedge \tau_k]}\left|\int_{\tau}^{\varkappa} \left(\varpi^n(r), (\sigma(r,u^n (r),\mathscr{L}_{u^n(r)})-\sigma(r,u (r),\mathscr{L}_{u(r)}))dW(r)\right)\right|\bigg) e^{\mathbf{k}_{10}(t-\tau)},
		\end{split}
	\end{align}
	where $\mathbf{k}_{10}=\mathbf{k}_{8}+\frac{2\mathfrak{C}^2}{\nu}k>0$.
	By taking the supremum of \eqref{NSE-disghjmea5.07} over $s\in [\tau,s_1]$ and then taking the expectation on both sides of the resulting expression, we obtain
	\begin{align}\label{NSE-disghjmea5.007}
		\begin{split}
			&\quad \mathbb{E}\left[\sup_{\varkappa\in [\tau,{s_1}\wedge \tau_k]} \|\varpi^n (\varkappa)\|_H^2\right]\\
			&\leq \Bigg(\mathbb{E}\left[\|\varpi^n_\tau\|_H^2\right]+\mathbf{k}_{9}\int_{\tau}^{{s_1} \wedge \tau_k} \mathbb{E}\left[\|\varpi^n(r)\|_H^2\right]dr+{2}\|g_n-g\|_{C_b(\mathbb{R},H)}^2+2\|h_n-h\|_{C_b(\mathbb{R},H)}^2\\
			&+2 \mathbb{E}\left[\sup_{\varkappa\in [\tau,{s_1}\wedge \tau_k]}\left|\int_{\tau}^\varkappa \left(\varpi^n(r),\sigma(r,u^n (r),\mathscr{L}_{u^n(r)})-\sigma(r,u (r),\mathscr{L}_{u(r)})\right)dW(r)\right|\right]\Bigg) e^{\mathbf{k}_{10}(t-\tau)}.
		\end{split}
	\end{align}
	For the stochastic integral term on the right-hand side of \eqref{NSE-disghjmea5.007}, by BDG's inequality and \eqref{NSE-dis-G5} we get
	{\small
	\begin{align*}
		&2\mathbb{E}\left[\sup_{\varkappa\in [\tau,{s_1}\wedge \tau_k]}\left|\int_{\tau}^\varkappa \left(\varpi^n(r), (\sigma(r,u^n (r),\mathscr{L}_{u^n(r)})-\sigma(r,u (r),\mathscr{L}_{u(r)}))dW(r)\right)\right|\right]\\
	\leq& 2\mathbb{E}\left[\left(\int_{\tau}^{s_1\wedge \tau_k} \|\sigma(r,u^n (r),\mathscr{L}_{u^n(r)})-\sigma(r,u (r),\mathscr{L}_{u(r)})\|^2_{\mathcal{L}_2(\ell^2,H)}\|\varpi^n(r)\|_H^2dr\right)^{1/2}\right]\\
		\leq& 2\sqrt{2}\|\kappa\|_{L^\infty(\mathcal{O})}\|L\|_{\ell^2}\sqrt{1+|\mathcal{O}|}\mathbb{E}\left[\sup_{\varkappa\in [\tau,s_1\wedge \tau_k]} \|\varpi^n (\varkappa)\|_H \left(\int_{\tau}^{s_1\wedge \tau_k} \left(\|\varpi^n(r)\|_H^2+\mathbb{E}\left[\|\varpi^n(r)\|_H^2\right]\right)dr\right)^{1/2}\right]\\
		\leq& \frac{1}{2}\mathbb{E}\left[\sup_{\varkappa\in [\tau,s_1\wedge \tau_k]} \|\varpi^n (\varkappa)\|_H^2\right]e^{-\mathbf{k}_{10}(t-\tau)}\\
		&+4\|\kappa\|^2_{L^\infty(\mathcal{O})}\|L\|_{\ell^2}^2(1+|\mathcal{O}|)e^{\mathbf{k}_{10}(t-\tau)}\int_\tau^{s_1} \mathbb{E}\left[\sup_{\varkappa\in [\tau,r\wedge \tau_k]}\|\varpi^n(\varkappa)\|_H^2\right]dr,
	\end{align*}
}
	which along with \eqref{NSE-disghjmea5.007} yields
	\begin{align}\label{NSE-disghjmea5.0007}
		\begin{split}
			&\mathbb{E}\left[\sup_{\varkappa\in [\tau,s_1\wedge \tau_k]} \|\varpi^n (\varkappa)\|_H^2\right]
			\leq \Bigg(2\mathbb{E}\left[\|\varpi^n_\tau\|_H^2\right]+4\|g_n-g\|_{C_b(\mathbb{R},H)}^2\\
			&\qquad \qquad +4\|h_n-h\|_{C_b(\mathbb{R},H)}^2+\mathbf{k}_{11}\int_{\tau}^{s_1} \mathbb{E}\left[\sup_{\varkappa\in [\tau,r\wedge \tau_k]}\|\varpi^n(\varkappa)\|_H^2\right]dr\Bigg) e^{2\mathbf{k}_{10}(t-\tau)},
		\end{split}
	\end{align}
	where $\mathbf{k}_{11}=2\mathbf{k}_{9}+8\|\kappa\|^2_{L^\infty(\mathcal{O})}\|L\|_{\ell^2}^2(1+|\mathcal{O}|)$.

	Applying Gronwall's lemma again to \eqref{NSE-disghjmea5.0007}, for any $\tau \in \mathbb{R}$, $t> \tau$ and $\tau<s<s_1<t$, $k>0$, we have 
	\begin{align*}
		&\mathbb{E}\left[\sup_{r\in [\tau,s_1\wedge \tau_k]} \|\varpi^n (r)\|_H^2\right]\\
		\leq&
		4e^{2\mathbf{k}_{10}(t-\tau)}\left(\mathbb{E}\left[\|\varpi^n_\tau\|_H^2\right]+\|g_n-g\|_{C_b(\mathbb{R},H)}^2+\|h_n-h\|_{C_b(\mathbb{R},H)}^2\right)e^{\mathbf{k}_{11}e^{2\mathbf{k}_{10}(t-\tau)}(t-\tau)}.
	\end{align*}
	Thus, we can obtain that for any $s_1\in [\tau, t]$,
\begin{align}\label{e5}
	\begin{split}
		&\mathbb{E}\left[\|u^{(g_n,h_n)}( s_1\wedge \tau_k,\tau, u_\tau^n)-u^{(g,h)}(s_1\wedge \tau_k,\tau, u_\tau)\|_H^2\right]\\
		\leq&
		\left(\mathbb{E}\left[\|\varpi^n_\tau\|_H^2\right]+\|g_n-g\|_{C_b(\mathbb{R},H)}^2+\|h_n-h\|_{C_b(\mathbb{R},H)}^2\right) 4e^{\left(2\mathbf{k}_{10}+\mathbf{k}_{11}e^{2\mathbf{k}_{10}(t-\tau)}\right)(t-\tau)}.
	\end{split}
\end{align}

Next, we prove that $u^{(g_n,h_n)}(s_1,\tau, u_{\tau}^n)$ converges in probability to $u^{(g,h)}(s_1,\tau, u_\tau)$ for any $\tau \in \mathbb{R}$, $s_1\in (\tau, t)$ with $t>\tau$. Note that for any $\delta>0$,
\begin{align}\label{dsfsdvgdbfb5.11}
	\begin{split}
		\mathbb{P}\left(\|\varpi^n (s_1)\|_H^2>\delta\right)
		\leq &\, \mathbb{P}\left(\sup_{r\in [\tau,s_1]}\|\varpi^n (r)\|_H^2>\delta\right)\\
		\leq &\, \mathbb{P}\left(\left\{\sup_{r\in [\tau,s_1]}\|\varpi^n (r)\|_H^2>\delta, \tau_k\geq s_1\right\}\right)\\
		&+\mathbb{P}\left(\left\{\sup_{r\in [\tau,s_1]}\|\varpi^n (r)\|_H^2>\delta, \tau_k<s_1\right\}\right)\\
		\leq &\, \mathbb{P}\left(\left\{\sup_{r\in [\tau,s_1]}\|\varpi^n (r\wedge \tau_k)\|_H^2>\delta\right\}\right)+\mathbb{P}\left(\tau_k<s_1\right).
	\end{split}
\end{align}
By Chebyshev's inequality, \eqref{e5} and Lemma \ref{NSE-dis-lemma004.2} we obtain that for any $\tau \in \mathbb{R}$, $s_1\in (\tau, t)$ with $t>\tau$,
\begin{align}\label{dsfsdvgdbfb5.12}
	\begin{split}
	&\mathbb{P}\left(\left\{\sup_{r\in [\tau,s_1]}\|\varpi^n (r\wedge \tau_k)\|_H^2>\delta\right\}\right)\leq \frac{1}{\delta}\mathbb{E}\left[\sup_{r\in [\tau,s_1]}\|\varpi^n (r\wedge \tau_k)\|^2_{H}\right]\\
	\leq& \frac{1}{\delta}\left(\mathbb{E}\left[\|\varpi^n_\tau\|_H^2\right]+\|g_n-g\|_{C_b(\mathbb{R},H)}^2+\|h_n-h\|_{C_b(\mathbb{R},H)}^2\right) 4e^{\left(2\mathbf{k}_{10}+\mathbf{k}_{11}e^{2\mathbf{k}_{10}(t-\tau)}\right)(t-\tau)}	
	\end{split}
\end{align}
and
\begin{align}\label{5.13}
\mathbb{P}\left(\tau_k<s_1\right)\leq \mathbb{P}\left(\int_\tau^{s_1} \|u(r)\|_V^2dr>k\right)\leq \frac{1}{k}\mathbb{E}\left[\int_\tau^{s_1} \|u(r)\|_V^2dr\right]\leq \frac{\mathscr{M}}{k},
\end{align}
Combining with \eqref{dsfsdvgdbfb5.11}-\eqref{5.13}, we obtain
%\begin{align}\label{5.14}
%	\begin{split}
%		&\mathbb{P}\left(\|\varpi^n (s_1)\|_H^2>\delta\right)\\
%		\leq& \frac{1}{\delta}\left(\mathbb{E}\left[\|\varpi^n_\tau\|_H^2\right]+\|g_n-g\|_{C_b(\mathbb{R},H)}^2+\|h_n-h\|_{C_b(\mathbb{R},H)}^2\right) 4e^{\left(2\mathbf{k}_{10}+\mathbf{k}_{11}e^{2\mathbf{k}_{10}(t-\tau)}\right)(t-\tau)}+\frac{\mathscr{M}}{k},
%	\end{split}
%\end{align}
\begin{align}\label{5.14}
%	\begin{split}
		&\mathbb{P}\left(\|u^{(g_n,h_n)}  (s_1,\tau,  u_{\tau}^n )
		- u^{(g,h)}  ( s_1,\tau, u_{\tau} )\|_H^2>\delta\right)\\
		\leq& \frac{1}{\delta}\left(\mathbb{E}\left[\|u_{\tau}^n - u_\tau\|_H^2\right]+\|g_n-g\|_{C_b(\mathbb{R},H)}^2+\|h_n-h\|_{C_b(\mathbb{R},H)}^2\right) 4e^{\left(2\mathbf{k}_{10}+\mathbf{k}_{11}e^{2\mathbf{k}_{10}(t-\tau)}\right)(t-\tau)}+\frac{\mathscr{M}}{k}, \notag
%	\end{split}
\end{align}
it is evident that both $\mathbf{k}_{10}$ and $\mathbf{k}_{11}$ are independent of $n$, $\tau$ and $t$.  

Thanks to $\mathbb{E}\left[\|u_{\tau}^n\|_H^4\right]\leq R$, the sequence $\{ u_{\tau}^n \}_{n=1}^\infty$ is uniformly integrable in $L^2(\Omega, H)$. Since
$u_{\tau}^n \to u_\tau$ $\mathbb{P}$-almost surely, it follows from Vitali’s theorem that $u_{\tau}^n \to u_\tau$ in   $L^2(\Omega, H)$. Moreover, by assumption we know that $(g_n,h_n)\rightarrow (g,h)$ in $\Sigma$.  Therefore, passing first to the limit as $n\rightarrow \infty$ and then as $k\rightarrow +\infty$ in \eqref{5.14}, we conclude for any $\tau \in \mathbb{R}$, $t>\tau$ and $s_1\in (\tau, t)$,
$$
\lim_{n\rightarrow +\infty}\mathbb{P}\left(\|\varpi^n (s_1)\|_H^2>\delta\right)=0,
$$
which implies that $u^{(g_n,h_n)}(s_1,\tau, u_{\tau}^n)$ converges in distribution to $u^{(g,h)}(s_1,\tau, u_\tau)$, we obtain the desired result. This completes the proof.
\end{proof}

It follows from Lemma \ref{NSE-disghjsss5.1} that the process $U^{(g,h)}(t, \tau)$ defined in \eqref{NSE-disghj5.2} is jointly continuous on bounded subsets of $\mathcal{P}_4(H) \times \Sigma$. We now proceed to demonstrate the existence of a uniform absorbing set for $U^{(g,h)}(t, \tau)$.
\begin{lemma}\label{NSE-disghlem5.1}
	Suppose that  $(H_1)-(H_3)$, \eqref{NSE-dis-def4.1} and \eqref{NSE-dis-def4.1*} hold. Denote by
	\begin{align}\label{absK}
	\mathcal{K}=\mathbb{B}_{\mathcal{P}_4(H)}\left(\sqrt[4]{\mathcal{M}_3}\right),	
	\end{align}
	 where $\mathcal{M}_3$ is from Lemma \ref{NSE-dis-lemma4.3}. Then
	\begin{center}
		$\mathcal{K}$ is a closed uniform absorbing set for $\{U^{(g,h)}(t, \tau)\}_{(g,h)\in \Sigma}$.
	\end{center}
	
\end{lemma}
\begin{proof}
	By \eqref{absK} and Lemma \ref{NSE-dis-lemma4.3}, we find that for each ${R}>0$, there exists $T=T(R)>0$ such that for any $t\geq T$ and $(g,h)\in \Sigma$, 
	$$
	U^{(g,h)}(t, 0)\mathbb{B}_{\mathcal{P}_4(H)}(R)\subseteq \mathcal{K}.
	$$
 Moreover, it is obvious that $\mathcal{K}$ is a closed subset of $\mathcal{P}_4(H)$. Therefore, we have that
	$\mathcal{K}$ is a closed
	uniform absorbing set
	for $\{U^{(g,h)}(t, \tau)\}_{(g,h)\in \Sigma}$.
	 This proof is finished.
\end{proof}

At last, we prove the existence and uniqueness of uniform measure attractors of problem \eqref{NSE-dis1.-01} in $(\mathcal{P}_4(H), d_{\mathcal{P} (H)})$.

\begin{theorem}
	Suppose that  $(H_1)-(H_3)$, \eqref{NSE-dis-def4.1} and \eqref{NSE-dis-def4.1*} hold. Then the family of processes $\{U^{(g,h)}(t, \tau)\}_{(g,h)\in \Sigma}$ associsted with \eqref{NSE-disghj5.2} possesses a unique uniform measure attractor $\mathscr{A}$ in $\mathcal{P}_4(H)$. This attractor is explicitly characterized by
	\[
	\mathscr{A} = \bigcup_{(g,h) \in \Sigma} \mathcal{K}_{(g,h)}(0),
	\]
	where \( \mathcal{K}_{(g,h)} \) denotes the kernel section of the process corresponding to the symbol \((g,h)\).
\end{theorem}
\begin{proof}
	First, the translation identity is established for the family of processes $\{U^{(g,h)}(t, \tau)\}_{(g,h)\in \Sigma}$. Moreover, Lemma \ref{NSE-disghjsss5.1} proves the joint continuity of this family on \( \mathcal{P}_4(H) \), and Lemma \ref{NSE-disghlem5.1} shows the existence of a closed uniform absorbing set $\mathcal{K}$ in 
	\( \mathcal{P}_4(H) \).
	 According to Theorem \ref{theo-2}, these properties imply that to conclude the existence of a uniform attractor, it remains to demonstrate that ${U^{(g,h)}(t, \tau)}_{(g,h)\in \Sigma}$ is uniformly asymptotically compact in $(\mathcal{P}_4(H), d_{\mathcal{P}(H)})$; i.e., the sequence $\{U^{(g_n,h_n)}(t_n,0)\mu_n\}$ admits a convergent subsequence in $\mathcal{P}_4(H)$ whenever $t_n \to \infty$ and $(\mu_n,(g_n,h_n))$ is bounded in \( \mathcal{P}_4(H) \times \Sigma\).

	Given $\xi_n\in L^4(\Omega,\mathscr{F}_{0}; H)$ with $\mathscr{L}_{\xi_n}=\mu_n$, we denote by $u^{(g_n,h_n)}(t_n,0,\xi_n)$ the solution of problem \eqref{NSE-dis1.-01} with initial data $\xi_n$ at initial time $0$, it is enough to prove that the sequence of distributions $\{\mathscr{L}_{u^{(g_n,h_n)}(t_n,0,\xi_n)}\}_{n=1}^\infty$ is tight in $H$. 
	From Lemma \ref{NSE-dis-lemma4.5} we see that there exists $N_1=N_1(R)\in \mathbb{N}$ such that for all $n\geq N_1$,  
	\begin{align}\label{NSE-disghjmea5.8}
		\mathbb{E}\left[\mathscr{G}(t_n,t_n-2,u^{(g_n,h_n)}(t_n-2,0,\xi_n))\| u^{(g_n,h_n)}(t_n,t_n-2,u^{(g_n,h_n)}(t_n-2,0,\xi_n))\|_V^2\right]\leq \mathcal{C}_1,
	\end{align}
	where $\mathcal{C}_1$ is a positive constant depending on $g_0,h_0$, but not on $\tau, \xi_n$ and $(g,h)\in \Sigma$.
	 
	Additionally, since $\gamma\in (0,\frac{1}{2})$ is small enough,  it follows from Lemma \ref{NSE-dis-lemma4.3} that there exist $N_2=N_2(R)\in \mathbb{N}$ and a constant $\mathcal{C}_2>0$ independent of $\tau, \xi_n$ and $(g,h)\in \Sigma$, such that for all $n\geq N_2$,
	{\small
	\begin{align}\label{NSE-disghjmea5.9}
%	\begin{split}
		&\int_{t_n-2}^{t_n} \mathbb{E}\left[\|u^{(g_n,h_n)}(s,t_n-2,u^{(g_n,h_n)}(t_n-2,0,\xi_n))\|_H^2\|u^{(g_n,h_n)}(s,t_n-2,u^{(g_n,h_n)}(t_n-2,0,\xi_n))\|_V^2\right]ds \nonumber\\
		\leq& e^{2}\int_{t_n-2}^{t_n} e^{-\gamma(t_n-s)}\mathbb{E}\left[\|u^{(g_n,h_n)}(s,t_n-2,u^{(g_n,h_n)}(t_n-2,0,\xi_n))\|_H^2\|u^{(g_n,h_n)}(s,t_n-2,u^{(g_n,h_n)}(t_n-2,0,\xi_n))\|_V^2\right]ds\nonumber\\
		\leq& \mathcal{C}_2.
%	\end{split}
	\end{align}
	%2-\gamma(t_n-s)>1
}
	By \eqref{NSE-disghjmea5.8}, \eqref{NSE-disghjmea5.9} and Markov's inequality we derive that there exists $N_3=\max\{N_1,N_2\}$ such that for  all $n\geq N_3$ and $\mathscr{R}>1$, 
	{\small
	\begin{align*}
		%\begin{split}
			&\mathbb{P}\left(\|u^{(g_n,h_n)}(t_n,0,\xi_n)\|_V^2>\mathscr{R}\right)\\
			&\leq  \mathbb{P}\left(\mathscr{G}(t_n,t_n-2,u^{(g_n,h_n)}(t_n-2,0,\xi_n))\|u^{(g_n,h_n)}(t_n,t_n-2,u^{(g_n,h_n)}(t_n-2,0,\xi_n))\|_V^2>\mathscr{R}^{1/2}\right)\\
			&+\mathbb{P}\left(\mathscr{G}^{-1}(t_n,t_n-2,u^{(g_n,h_n)}(t_n-2,0,\xi_n))>\mathscr{R}^{1/2}\right)\\
			&\leq \mathbb{P}\left(\mathscr{G}(t_n,t_n-2,u^{(g_n,h_n)}(t_n-2,0,\xi_n))\|u^{(g_n,h_n)}(t_n,t_n-2,u^{(g_n,h_n)}(t_n-2,0,\xi_n))\|_V^2>\mathscr{R}^{1/2}\right)\\
			&+\mathbb{P}\left(\int^{t_n}_{t_n-2}\|u^{(g_n,h_n)}(s,t_n-2,u^{(g_n,h_n)}(t_n-2,0,\xi_n))\|_H^2\|u^{(g_n,h_n)}(s,t_n-2,u^{(g_n,h_n)}(t_n-2,0,\xi_n))\|_V^2ds>\frac{2\nu^3}{27\mathfrak{C}^4}\ln{\mathscr{R}}\right)\\
			&\leq \frac{\mathbb{E}\left[\mathscr{G}(t_n,t_n-2,u^{(g_n,h_n)}(t_n-2,0,\xi_n))\|u^{(g_n,h_n)}(t_n,t_n-2,u^{(g_n,h_n)}(t_n-2,0,\xi_n))\|_V^2\right]}{\mathscr{R}^{1/2}}\\
			& +\frac{27\mathfrak{C}^4}{2\nu^3\ln{\mathscr{R}}}\int^{t_n}_{t_n-2}\mathbb{E}\left[\|u^{(g_n,h_n)}(s,t_n-2,u^{(g_n,h_n)}(t_n-2,0,\xi_n))\|_H^2\|u^{(g_n,h_n)}(s,t_n-2,u^{(g_n,h_n)}(t_n-2,0,\xi_n))\|_V^2\right]ds\\
			&\leq \frac{\mathcal{C}_1}{\mathscr{R}^{1/2}}+\frac{27\mathfrak{C}^4\mathcal{C}_2}{2\nu^3\ln{\mathscr{R}}}\rightarrow 0, \quad \text{as}\quad \mathscr{R}\rightarrow \infty.
		%\end{split}
	\end{align*}
}

	Therefore, for any $\delta>0$, there exists $\hat{\mathscr{R}}=\hat{\mathscr{R}}(\delta)>0$ such that for all $\xi_n\in L^4(\Omega,\mathscr{F}_{0}, H)$ with $\mu_n\in \mathcal{K}$ and for any $n\geq N_3$,
	\begin{align*}
		\mathbb{P}\left(\|u^{(g_n,h_n)}(t_n,0,\xi_n)\|^2_V>\hat{\mathscr{R}}\right)<\delta,
	\end{align*}
	which, together with the compact embedding $V\hookrightarrow H$, shows that the sequence  $\{\mathscr{L}_{u^{(g_n,h_n)}(t_n,0,\xi_n)}\}_{n=1}^\infty$ is tight in $H$. This further implies that there exist a probability measure $\hat{\mu}\in \mathcal{P}(H)$ and a subsequence of $\{\mathscr{L}_{u^{(g_n,h_n)}(t_n,0,\xi_n)}\}_{n=1}^\infty$ (not relabel) such that
	\begin{align}\label{NSE-disghjmea5.11}
		\mathscr{L}_{u^{(g_n,h_n)}(t_n,0,\xi_n)} \rightarrow \hat{\mu} \text{~weakly}.
	\end{align}  
	Finally, we prove that $\hat{\mu}\in \mathcal{P}_4(H)$. Let $\mathcal{K}$ be the closed uniform absorbing set of $\{U^{(g,h)}(t, \tau)\}_{(g,h)\in \Sigma}$ given by \eqref{absK}. Then, there exists $N_4\in \mathbb{N}$ such that for all $n\geq N_4$,
	\begin{align}\label{NSE-disghjmea5.12}
		\mathscr{L}_{u^{(g_n,h_n)}(t_n,0,\xi_n)}\in \mathcal{K}.
	\end{align}
	Observe that $\mathcal{K}$ is closed with respect to the weak topology of $\mathcal{P}_4(H)$, then by \eqref{NSE-disghjmea5.11} and \eqref{NSE-disghjmea5.12} we obtain $\hat{\mu}\in\mathcal{K}$, from which we have $\hat{\mu}\in \mathcal{P}_4(H)$.  This proof is finished.
\end{proof}

\section*{Declarations}

\subsection*{Conflict of Interest}
The authors declare no conflict of interest.

\subsection*{Data Availability}
No new data were created or analysed in the article.

\end{document}